\documentclass[11pt]{amsart}
\usepackage{amsmath}
\usepackage{libertine}
\usepackage[libertine]{newtxmath}
\usepackage{mathrsfs}

\usepackage[T1]{fontenc}
\usepackage{hyperref}
\usepackage{url}
\usepackage{cite,amsmath,amstext,amsbsy,bm} 
\usepackage{datetime} 
\usepackage{colortbl}
\usepackage[dvipsnames]{xcolor}
\usepackage{mathtools}
\usepackage{euscript}
\usepackage[all]{xy}

\usepackage{graphicx} 
\usepackage[a4paper,text={360pt,615pt},centering]{geometry} 
\usepackage{comment}  
\usepackage{cite}  
\usepackage{thmtools}
\usepackage{enumitem}
\usepackage{nameref}
\usepackage{cleveref}
\usepackage{calc}
\usepackage{interval}
\usepackage{manfnt}
\usepackage{tikz-cd}  
\usepackage[utf8]{inputenc}
\usepackage{newunicodechar} 

\usepackage[hyperpageref]{backref}

\usepackage{csquotes}

\theoremstyle{plain}
\newtheorem{thmx}{Theorem}
\renewcommand{\thethmx}{\Alph{thmx}} 
\newtheorem{thm}{Theorem}[section]  
\newtheorem{lem}[thm]{Lemma}
\newtheorem*{pri}{Principle}

\newtheorem{prop}[thm]{Proposition}
\newtheorem{cor}[thm]{Corollary}

\newtheorem{corx}[thmx]{Corollary}
\newtheorem{claim}[thm]{Claim}

\theoremstyle{definition}
\newtheorem{dfn}[thm]{Definition}

\newtheorem{assumpsion}[thm]{Assumption}

\theoremstyle{remark}
\newtheorem{rem}[thm]{Remark}

\numberwithin{equation}{subsection}  

\theoremstyle{plain}
\newlist{thmlist}{enumerate}{1}
\setlist[thmlist]{wide = 0pt, labelwidth = 2em, labelsep*=0em, itemindent = 0pt, leftmargin = \dimexpr\labelwidth + \labelsep\relax, noitemsep,topsep = 1ex, font=\normalfont, label=(\roman*), ref=\thethm.(\roman{thmlisti})}

\addtotheorempostheadhook[thm]{\crefalias{thmlisti}{thm}}

\addtotheorempostheadhook[assumpsion]{\crefalias{thmlisti}{assumption}}

\addtotheorempostheadhook[proposition]{\crefalias{thmlisti}{proposition}}

\addtotheorempostheadhook[dfn]{\crefalias{thmlisti}{dfn}}

\addtotheorempostheadhook[lem]{\crefalias{thmlisti}{lem}}
\addtotheorempostheadhook[main]{\crefalias{thmlisti}{main}}

\addtotheorempostheadhook[rem]{\crefalias{thmlisti}{rem}}

\newlist{thmenum}{enumerate}{1} 
\setlist[thmenum]{wide = 0pt, labelwidth = 2em, labelsep*=0em, itemindent = 0pt, leftmargin = \dimexpr\labelwidth + \labelsep\relax, noitemsep,topsep = 1ex, font=\normalfont, label=(\roman*), ref=\thethmx.(\roman{thmenumi})}
\crefalias{thmenumi}{thmx}

\DeclareMathOperator{\Ima}{Im}

\makeatletter
\newsavebox{\@brx}
\newcommand{\llangle}[1][]{\savebox{\@brx}{\(\m@th{#1\langle}\)}%
	\mathopen{\copy\@brx\kern-0.5\wd\@brx\usebox{\@brx}}}
\newcommand{\rrangle}[1][]{\savebox{\@brx}{\(\m@th{#1\rangle}\)}%
	\mathclose{\copy\@brx\kern-0.5\wd\@brx\usebox{\@brx}}}
\makeatother

\crefname{lem}{Lemma}{Lemmas}
\crefname{thm}{Theorem}{Theorems}
\crefname{proposition}{Proposition}{Propositions}
\crefname{dfn}{Definition}{Definitions}
\crefname{rem}{Remark}{Remarks}
\crefname{cor}{Corollary}{Corollaries}
\crefname{corx}{Corollary}{Corollaries}
\crefname{problem}{Problem}{Problems}
\crefname{thmx}{Theorem}{Theorems}
\crefname{claim}{Claim}{Claims}
\crefname{assumption}{Assumption}{Assumptions}
\crefname{main}{Main Theorem}{Main Theorems}

\def\ep{\varepsilon}

\def\rank{{\rm rank}}

\newcommand{\diae}{{}^\diamond\! E}

\makeatletter
\newcommand*{\rom}[1]{\expandafter\@slowromancap\romannumeral #1@}
\makeatother

\makeatletter     
\newcommand{\crefnames}[3]{%
	\@for\next:=#1\do{%
		\expandafter\crefname\expandafter{\next}{#2}{#3}%
	}%
}
\makeatother

\crefnames{part,chapter,section}{\S}{\S\S}

\newcommand{\sA}{\mathscr{A}}

\newcommand{\sC}{\mathscr{C}}

\newcommand{\sE}{\mathscr{E}}
\newcommand{\sF}{\mathscr{F}}

\newcommand{\sH}{\mathscr{H}}

\newcommand{\sL}{\mathscr{L}}

\newcommand{\sO}{\mathscr{O}}

\newcommand{\cE}{\mathcal E}
\newcommand{\cX}{\mathcal X}

\newcommand{\cO}{\mathcal O}

\newcommand{\kL}{\mathfrak L} 

\newcommand{\bC}{\mathbb{C}}
\newcommand{\bD}{\mathbb{D}}

\newcommand{\bH}{\mathbb{H}}

\newcommand{\bR}{\mathbb{R}}

\def\db{\bar{\partial}}

 \def\d{\partial} 
\def\hess{\sqrt{-1}\partial\overline{\partial}}

\def\sn{\sqrt{-1}}

\def\dls{\mathbb{D}_h^{\lambda\star}}

\def\nak{\text{\tiny  Nak}}

\def\gri{\text{\tiny  Gri}}

\def\loc{\, \text{\tiny  loc}}

\def\End{\text{\small  End}}

\def\Dom{\text{\small  Dom}}

\def\vol{\text{\small  Vol}}

\def\Herm{\text{\small  Herm}}

\def\dls{\mathbb{D}_h^{\lambda\star}}
\def\lmd{\lambda}

\makeatletter
\hypersetup{
	pdfauthor={\authors},
	pdftitle={\@title},
	pdfsubject={\@subjclass},
	pdfkeywords={\@keywords},
	pdfstartview={Fit},
	pdfpagelayout={TwoColumnRight},
	pdfpagemode={UseOutlines},
	bookmarks,
	colorlinks,
	linkcolor=Sepia,
	citecolor=OliveGreen,
	urlcolor=WildStrawberry}
\makeatother

\begin{document} 
	
	\title[Vanishing Theorem for tame  harmonic bundles]{Vanishing Theorem for tame    \\ harmonic bundles  via $L^2$-cohomology}

	\author{Ya Deng}  
	\address{CNRS, Institut \'Elie Cartan de Lorraine, Universit\'e de Lorraine, F-54000 Nancy,
		France.}
	
	\email{ya.deng@univ-lorraine.fr \quad ya.deng@math.cnrs.fr}
	
	\urladdr{https://ydeng.perso.math.cnrs.fr} 
	
		\author{Feng Hao} 
		\address{Department of Mathematics, KU LEUVEN,
Celestijnenlaan 200B, bus 2400
B-3001 Leuven, Belgium}
			\email{feng.hao@kuleuven.be}  
	 \urladdr{https://www.kuleuven.be/wieiswie/nl/person/00133186}		
	\date{\today} 
	\begin{abstract} 
  Using $L^2$-methods, we prove a    vanishing theorem  for tame  harmonic bundles over quasi-compact K\"ahler manifolds   in a very general setting. As a special case, we give a completely new proof of the Kodaira type vanishing theorems for  Higgs bundles due to Arapura.  To prove our vanishing theorem, we construct a fine resolution of  the Dolbeault complex  for tame  harmonic bundles via the complex  of sheaves of $L^2$-forms, and we  establish the H\"ormander $L^2$-estimate and solve $(\bar{\partial}_E+\theta)$-equations for Higgs bundles  $(E,\theta)$. 
\end{abstract} 
\subjclass[2010]{14C30, 14F17, 	32L20, 14J60}
\keywords{tame harmonic bundle, (parabolic) Higgs bundle, vanishing theorem, H\"ormander $L^2$-estimate,  Bochner technique, $L^2$-cohomology,  Simpson-Mochizuki correspondence}
	\maketitle
	
\tableofcontents
\section{Introduction} 
\subsection{Main result}
Let $(X,\omega)$ be a compact K\"ahler manifold and let $D$ be a simple normal crossing divisor on  {$X$}. Let $(E,\theta,h)$ be a tame   harmonic bundle over $X-D$ so that $\theta$ has nilpotent residues on $D$(see \cref{sec:tame} for the precise definition), and let $\diae$ be the subsheaf   of $\iota_*E$ consisting of sections whose norms with respect to $h$ have sub-polynomial growth  {(see \cref{sec:prolongation})}, where $\iota:X-D\hookrightarrow X$ is the inclusion. By Simpson-Mochizuki, $\diae$ is a locally free coherent sheaf, and  $(E,\theta)$ extends to a logarithmic Higgs bundle
	$$\theta:\diae\to \diae\otimes \Omega_X^1(\log D)$$
such that
$$\theta\wedge \theta=0.$$
 We refer to \cref{sec:prolongation} for more details. 
 
In this paper, we prove the following vanishing theorem.
\begin{thmx}[=\cref{thm:main}]\label{main}
	Let $(X,\omega)$ be a compact  K\"ahler manifold of dimension $n$, and let $D$ be a simple normal crossing divisor on $X$. Let $(E,\theta)$ be  a  tame   harmonic bundle    on $X-D$ so that $\theta$ has nilpotent residues on $D$, and let $(\diae,\theta)$ be the extension of $(E,\theta)$ on $X$ as  introduced above.   Let $L$ be a {holomorphic} line bundle on $X$ equipped with a smooth {Hermitian} metric $h_L$ 
	so that its curvature $\sqrt{-1}R(h_L)\geq 0$ and has at least $n-k$ positive eigenvalues {at every point on $X$ as a real (1,1)-form}. Let {$B$} be a nef line bundle on $X$.     Then
	for the following (Dolbeault) complex of sheaves
	\begin{align}\label{eq:dol}
	{\rm Dol}(\diae,\theta):=\diae\xrightarrow{\wedge\theta}\diae\otimes \Omega^1_X(\log D) \xrightarrow{\wedge\theta}\cdots\xrightarrow{\wedge\theta}\diae\otimes \Omega^n_X(\log D) 
	\end{align}
	the hypercohomology
	$$
	\mathbb{H}^i\big(X, {\rm Dol}(\diae,\theta)\otimes L\otimes {B} \big)=0
	$$
	for any $i>n+k$.
\end{thmx}
\cref{main} seems new even if the  tame harmonic bundle $(E,\theta,h)$ comes from a complex  variation of polarized Hodge structures {over $X-D$}. It indeed interpolates the Kodaira-Akizuki-Nakano type vanishing theorems for nilpotent Higgs bundles    \cite[Theorem 1]{Ara19}   by Arapura (in the case that $L$ is ample, see \cref{cor:main}), and the log Girbau vanishing theorem by Huang-Liu-Wan-Yang \cite[Corollary 1.2]{HLWY16} (in the case that $(E,\theta)=(\sO_{X-D},0)$, see \cref{rem:Yang}). We   stress here that our proof of \cref{main} is  essentially self-contained (in particular we do not apply the deep Simpson-Mochizuki correspondence) and  is purely  in characteristic $0$ (since we are working on K\"ahler manifolds), comparing to the celebrated vanishing theorem by  Arapura   \cite{Ara19}  whose proof is in characteristic $p$ (see \cref{sec:previous} for more details). 
The main technique  in the proof of \cref{main} is a new application of $L^2$-methods to tame harmonic bundles, and we hope that it can bring some new input in the study of $L^2$-cohomology for Higgs bundles. Let us also mention a few byproducts of our proof: we construct explicitly complexes of  sheaves of $L^2$-forms for tame   Higgs bundles which are quasi-isomorphic to the Dolbeault complexes \eqref{eq:dol} (see \cref{thm:quasi}) in a similar manner (but using  different metric) as \cite{Zuc79}  in which Zucker did this for  variation of polarized Hodge structures over a quasi-projective curve; we also establish the H\"ormander $L^2$-estimate and solvability criteria for $(\db_E+\theta)$-equations for Higgs bundles $(E,\theta)$ (see \cref{thm:L2,cor:AN}).

If we apply the Simpson-Mochizuki correspondence \cite{Sim90,Moc06} for parabolic Higgs bundles on projective manifolds to  \cref{main}, we can obtain the following vanishing theorem for parabolic Higgs bundles whose Higgs fields have nilpotent residues (see \cref{sec:tame} for the precise definition).

\begin{corx}[=\cref{cor:main}]\label{corx}
Let $X$	be a complex projective manifold of dimension $n$, and let $D$ be simple normal crossing divisor on $X$. Let $(E, {}_{\bm{a}}E,\theta)$ be a locally abelian poly-stable parabolic Higgs bundle on $(X,D)$ with trivial parabolic Chern classes so that $\theta$ has nilpotent residues on $D$.   Let $L$ be a {holomorphic} line bundle on $X$ equipped with a smooth  {Hermitian} metric $h_L$ 
	so that its curvature $\sqrt{-1}R(h_L)\geq 0$ and has at least $n-k$ positive eigenvalues  {at every point on $X$ as a real (1,1)-form}. Let {$B$} be a nef line bundle on $X$.   Then  for the weight 0 filtration $\diae$  of $(E,{}_{\bm{a}}E,\theta)$, one has
	$$
	\mathbb{H}^i\big(X, {\rm Dol}(\diae,\theta)\otimes L\otimes {B} \big)=0
	$$
	for any $i>n+k$.   
\end{corx}

For the notions in  \cref{corx} we refer to \cref{sec:tame}, {\cref{sec:parahiggs} and \cref{sec:parabolic}} for more details.



\subsection{Idea of the proof}
Let us briefly explain the main idea  of our proof of \cref{main}. We first construct a complex of $L^2$ fine sheaves for the tame  harmonic bundle  $(E,\theta,h)$ whose Higgs field $\theta$ has nilpotent residues on $D$, which is quasi-isomorphic to the Dolbeault complex
\begin{equation}\label{dia:dol}
\begin{tikzcd}
{\rm Dol}(\diae,\theta):=\diae \arrow[r,"\theta"] &\diae\otimes \Omega^1_X(\log D)  \arrow[r,"\theta"]&\cdots  \arrow[r,"\theta"] &\diae\otimes \Omega^n_X(\log D) 
\end{tikzcd}
\end{equation}

 For the given K\"ahler metric $\omega$ on $X$ (we denote the restricted K\"ahler form $\omega|_{X-D}$ again by $\omega$ over $X-D$) and a smooth Hermitian metric $g$ for $E$ over $X-D$,  we let $\kL^{m}_{(2)}(X,E)_{g,\omega}$ be the sheaf on $X$ of germs of  $E$-valued
$m$-forms $\sigma$ with measurable coefficients so that $\lvert \sigma\rvert_{g,\omega}^2$ is locally integrable and $(\db+\theta)(\sigma)$ exists weakly as  a locally $L^2$, $E$-valued
$(m+1)$-form. Here the $L^2$ norms $\lvert \sigma\rvert_{g,\omega}^2$ are induced by $\omega$ on differential
forms and by $g$ on elements in $E$.    Since $(\db+\theta)^2=0$, it thus gives rise to a complex of fine sheaves
\begin{align} \label{eq:L^2}
\kL^{0}_{(2)}(X,E)_{g,\omega}\stackrel{\db+\theta}{\to} \cdots\stackrel{\db+\theta}{\to} \kL^{2n}_{(2)}(X,E)_{g,\omega}
\end{align}
As the harmonic metric $h$ is a \emph{canonical} metric on   $E$, it is quite natural to  make the choice that $g$ is the harmonic metric $h$. Also, we replace the k\"ahler form $\omega$ by a Poincar\'e-type metric $\omega_P$ over $X-D$ as
  \cite{Zuc79,CKS87,KK87}. However, even for the case when $(E,\theta)$ comes from a  variation of polarized Hodge structures over $X-D$, it turns out to be  a quite difficult problem that $(\kL_{(2)}^{\bullet}(X,E)_{h,\omega_P},\db+\theta)$ is quasi-isomorphic to $ {\rm Dol}(\diae,\theta)$, and one essentially cannot avoid  the delicate norm estimate for Hodge metrics near $D$ in \cite{Sch73,Kas85,CKS86} (see \emph{e.g.} \cite{Zuc79,JYZ07}). In this paper, we make a slight perturbation {$h_{\bm{a},N}$} of  the harmonic metric $h$ (see \cref{lem:perturbe} for more details)   as  \cite[\S 4.5.3]{Moc02}  so that {$h_{\bm{a},N}$} will degenerate mildly, albeit the norm of harmonic metric $h$ for $\diae$ is  of sub polynomial growth. {Also, we slightly perturb the K\"ahler metric $\omega$ on $X-D$ into a complete K\"ahler metric $\omega_{\bm{a},N}$ (see \cref{lem:choice}), which is mutually bounded with the Poincar\'e metric $\omega_P$ near the divisor $D$.} This construction indeed brings us several advantages (among others): we can prove that $(\kL_{(2)}^{\bullet}(X,E)_{{h_{\bm{a},N},\omega_{\bm{a},N}}},\db+\theta)$ is indeed quasi-isomorphic to ${\rm Dol}(\diae,\theta)$, and the negative contribution of the curvature $(E,\theta,{h_{\bm{a},N}})$ is small enough which can be absorbed completely by the curvature $\sn R(h_L)$ of any (partially) positive metrized line bundle $(L,h_L)$.

Thus we have the following $L^2$ fine resolution of ${\rm Dol}(\diae,\theta)\otimes L$

\begin{align}\label{eq:L^2 global}
(L_{(2)}^{\bullet}(X-D,E\otimes L|_{X-D})_{{h_{\bm{a},N}}\cdot h_L,{\omega_{\bm{a},N}}},D''),
\end{align}
 where $D'':=\db_{E\otimes L}+\theta\otimes \vvmathbb{1}_L$ satisfying $D''^2=0$ (Here we assume $M=\sO_X$ for simplicity). We then reduce the proof of  \cref{main} to the vanishing of $i$-th cohomology of the complex of global sections of \eqref{eq:L^2 global} for  $i>\dim X+k$.  To prove this, we first  generalize the $L^2$-estimate  by H\"ormander, Andreotti-Vesentini, Skoda, Demailly  and others to Higgs bundles.  Roughly speaking, we prove that under certain curvature conditions for Higgs bundles $(E,\theta)$ over $X-D$, we can solve the $D''$-equation as the $\db$-equation in a similar way (see \cref{thm:L2,cor:AN}). We then choose the perturbation {$h_{\bm{a},N}$} of $h$ carefully so that such required curvature condition can be fulfilled and it enables us to prove the vanishing result for the $L^2$-cohomology of \eqref{eq:L^2 global}.  This idea of solving $D''$-equation for Higgs bundles using $L^2$-method  seems a new ingredient as we are aware of.




\subsection{Previous results} \label{sec:previous}
For $X$ a complex projective manifold with a simple normal crossing divisor $D$, Arapura \cite{Ara19} gives a vanishing theorem for semistable Higgs bundles $(E, \theta)$ over $X-D$ with trivial parabolic structure, trivial Chern classes and nilpotent Higgs field $\theta$.  In the spirit of the algebraic proof of the Kodaira vanishing theorem by Deligne-Illusie \cite{DL87}, the proof of Arapura's vanishing theorem is reduced to the mod $p$-setting and boils down to a periodic sequence of Higgs bundles $(E_i, \theta_i):=B^i(E, \theta)$  through an operator $B$ raised from the absolute Frobenius morphism, which is due to    Lan-Sheng-Yang-Zuo \cite{LSZ19,LSYZ13} and Langer \cite{Lan15}.  The dimension of the cohomology $\mathbb{H}^i(X, {\rm Dol}(E_i, \theta_i)\otimes L^{p^i})$ is non-decreasing for $\{(E_i, \theta_i)\}$ and ample line bundle $L$, then Arapura's vanishing theorem follows from Serre's vanishing theorem. With his vanishing theorem, Arapura reproves the Saito's vanishing theorem (see, e.g.  Popa \cite{Pop16}) for variation of polarized Hodge structures with unipotent monodromy on the complement of a normal crossing divisor on any complex projective manifold. 
\subsection*{Acknowledgements}
The first author would like to thank Professors Olivier Biquard, Junyan Cao, Jean-Pierre Demailly, Takuro Mochizuki, Carlos Simpson for answering his questions, and Professors Jian Xiao, Xiaokui Yang, Kang Zuo for their comments on this paper.  He also would like to thank the support and perfect working condition provided by   IH\'ES. The second author is supported by grant G097819N of Nero Budur and  1280421N from the Research Foundation Flanders (FWO). This work was started while both authors were participating the workshop on \enquote{mixed Hodge modules and Hodge ideals} at the university of Angers on 1-5 April 2019. We 
would like to thank  the organizers   for their hospitality. Last but not least, we are grateful to the referee for his/her careful readings and very helpful comments to improve this manuscript.

	\section*{Notations and conventions} 
\begin{itemize}[wide = 0pt, labelwidth = 2em, labelsep*=0em, itemindent = 0pt, leftmargin = 0pt, noitemsep,topsep = 1ex, font=\normalfont ]
	\item  A couple $(E,h)$ is a \emph{Hermitian vector bundle} on a complex manifold $X$ if $E$  {is} a holomorphic vector bundle on $X$ equipped with a smooth hermitian metric $h$.  $\db_E$ denotes the complex structure of $E$, and we sometimes simply write $\db$ if no confusion arises.
	\item Two hermitian metrics $h$ and $\tilde{h}$ of a holomorphic vector bundle on $X$ are \emph{mutually bounded} if $C^{-1}h\leq \tilde{h}\leq Ch$ for some constant $C>0$, and we shall denote by $h\sim h'$.
	\item For a hermitian vector bundle $(E,h)$ on a complex manifold,  $R(E, h)$ or simply $R(h)$ denotes its Chern curvature. 
	\item $\Delta$ denotes the unit disk in $\bC$.
	\item The complex manifold $X$ in this paper are always assumed to be connected and of dimension $n$.
	\item Throughout the paper we always work over the complex number field $\bC$.
\end{itemize}

\section{Technical preliminary}

\subsection{Higgs bundle and tame harmonic bundle}\label{sec:tame}
	In this section we   recall the   definition of Higgs bundles and tame harmonic bundles. We refer the readers to  \cite{Sim88,Sim90,Sim92,Moc02,Moc07} for further details.
\begin{dfn}\label{Higgs}
	Let $X$ be a complex manifold. A \emph{Higgs bundle} on $X$  is a pair $(E,\theta)$ where $E$ is a holomorphic vector bundle with $\db_E$ its complex structure, and $\theta:E\to E\otimes \Omega^1_X$ is a holomorphic one form with value in $\End(E)$, say \emph{Higgs field},  satisfying $\theta\wedge\theta=0$.
\end{dfn}

 Let  $(E,\theta)$ be a Higgs bundle  over a complex manifold $X$.  Write $D'':=\db_E+\theta$. Then $ D''^2=0$. 
 Suppose $h$ is a smooth hermitian metric of $E$.  Denote by $\mathbb{\partial}_h+\db_E$  the Chern connection with respect to $h$, and $\theta^*_h$ be the adjoint of $\theta$ with respect to $h$. Write $D_h':= {\partial}_h+\theta_h^*$. The metric $h$ is  \emph{harmonic} if the operator $D_h:=D_h'+D''$
is integrable, that is, if  $D_h^2=0$.
\begin{dfn}[Harmonic bundle] A harmonic bundle on a complex manifold $X$ is
	a Higgs bundle $(E,\theta)$ endowed with a  harmonic metric $h$.
\end{dfn}
Let $X$ be an $n$-dimensional complex manifold, and let $D$ be
a simple normal crossing divisor on $X$.
\begin{dfn}(Admissible coordinate)\label{def:admissible} Let $p$ be a point of $X$, and assume that $\{D_{j}\}_{ j=1,\ldots,\ell}$ 
	be components of $D$ containing $p$. An \emph{admissible coordinate} around $p$
	is the tuple $(U;z_1,\ldots,z_n;\varphi)$ (or simply  $(U;z_1,\ldots,z_n)$ if no confusion arises) where
	\begin{itemize}
		\item $U$ is an open subset of $X$ containing $p$.
		\item there is a holomorphic isomorphism   $\varphi:U\to \Delta^n$ so that  $\varphi(D_j)=(z_j=0)$ for any
		$j=1,\ldots,\ell$.
	\end{itemize} 
We shall write $U^*:=U-D$,   $U(r):=\{z\in U\mid |z_i|<r, \, \forall i=1,\ldots,n\}$ and $U^*(r):=U(r)\cap U^*$.  
\end{dfn}
For any  harmonic bundle $(E,\theta,h)$, let $p$ be any point of X, and $(U;z_1,\ldots,z_n)$  be an admissible coordinate around  $p$. On $U$, we have the description:
\begin{align}\label{eq:local}
\theta=\sum_{j=1}^{\ell}f_jd\log z_j+\sum_{k=\ell+1}^{n}g_kdz_k
\end{align}

\begin{dfn}[Tameness]\label{def:tameness}
 Let $t$ be a formal variable. We have the polynomials $\det (f_j-t)$,  and $\det (g_k-t)$, whose coefficients are holomorphic functions defined over $U^*$. When the functions can be extended
	to the holomorphic functions over $U$, the harmonic bundle is called \emph{tame} at $p$.  A harmonic bundle is \emph{tame} if it is tame at each point.
\end{dfn}

\begin{dfn}[Nilpotent residues]\label{def:nilpotency}
Let $(E, \theta)$ be a Higgs bundle on $X-D$. We say that $\theta$ has \emph{nilpotent residues on $D$} if for each component $D_j$ of $D$ and any  point $p\in D_j$ one has $\det(f_j-t)|_{U\cap D_j}=(-t)^{\rank E}$.   
\end{dfn}
 
	\begin{rem}\label{rmk:nilpotency-compare}
	One should notice that the above definition   introduced in \cite[p 435]{Moc02}  is more general than that in \cite[Theorem 1]{Ara19}, where the nilpotency of Higgs field $\theta$ is defined to be the local matrix of $\theta$ is nilpotent.  We   refer   readers to \cite{Ara19} for more details.
	\end{rem}

Recall that the Poincar\'e metric $\omega_P$ on $(\Delta^*)^\ell\times \Delta^{n-\ell}$ is described as 
$$
\omega_P=\sum_{j=1}^{\ell}\frac{\sqrt{-1}dz_j\wedge d\bar{z}_j}{|z_j|^2(\log |z_j|^2)^2}+\sum_{k=\ell+1}^{n}\frac{\sqrt{-1}dz_k\wedge d\bar{z}_k}{(1-|z_k|^2)^2}.
$$
Note that 
$$
\omega_P=-\hess \log \big(\prod_{j=1}^{\ell}(-\log |z_j|^2)\cdot \prod_{k=\ell+1}^{n}(1-|z_k|^2)\big).
$$

For the tame  harmonic bundle so that the Higgs field has nilpotent residues, we have the following crucial norm estimate for Higgs field $\theta$. The   one dimensional case  is due to Simpson \cite[Theorem 1]{Sim90} and Mochizuki \cite[Proposition 4.1]{Moc02} in general. 
\begin{thm} \label{thm:moc}
	Let $(E,\theta,h)$ be a tame  
	harmonic bundle on $X-D$ so that  $\theta$ has nilpotent residues on $D$. Let $f_j,g_k$ be the matrix-valued holomorphic  functions as in \cref{def:tameness}.  Then there exists a positive constant $C>0$ satisfying that
	\begin{align*} 
	&|f_j|_h\leq \frac{C}{-\log |z_j|^2},  \quad &\mbox{for} \quad j=1,\ldots,\ell;\\
	&|g_k|_h \leq C, \quad &\mbox{for} \quad  k=\ell+1,n.
	\end{align*} 
	In other words, the norm 
	$$
	| \theta|_{h,\omega_P}\leq C  
	$$
holds over  $U^*(r)$	for some constant $C>0$ and $0<r< 1$. \qed
\end{thm}

\subsection{Curvature property of Higgs bundles}
Suppose now $(E, \theta)$ is a   Higgs bundle of rank $r$ equipped with a Hermitian metric $h$ over a  K\"ahler manifold $(X,\omega)$ of dimension $n$ . 

We make the following assumption for $(E,\theta,h)$ throughout this section .
\begin{assumpsion}\label{assum}
	$\db_E\theta_h^*=0$.
\end{assumpsion}
 Consider the connection $D_h:=D_h'+D''$ (see the paragraph after \cref{Higgs}).  \Cref{assum} is equivalent to that  $\partial_h\theta=0$.  Hence one has the curvature
\begin{align}\label{eq:assumption}
F(h)\coloneqq D_h^2=[D_h',D'']=R(h)+[\theta,\theta^*] \in A^{1,1}(X,\End(E)),
\end{align}
where $R(h)\coloneqq (\partial_h+\db_E)^2$. Moreover, one can easily see that $(\sqrt{-1}F(h))^{*}=\sqrt{-1}F(h)$. In other words, $\sqrt{-1}F(h)$ is a $(1,1)$-form with $\Herm(E)$-value, where $\Herm(E)$ is the hermitian endomorphism of $(E,h)$. 

By Simpson \cite{Sim88}, one has the following K\"ahler identities:
\begin{align}\label{eq:kahler}
\sqrt{-1}[\Lambda_\omega,D'']=(D_h')^*\\\label{eq:kahler2}
\sqrt{-1}[\Lambda_\omega,D_h']=-(D'')^* 
\end{align}
where $(D_h')^*$ and $(D'')^*$ are the formally adjoint operators of $D_h'$ and $D''$ with  respect to $h$ and $\omega$, and $\Lambda_{\omega}$ is the adjoint operator of $\wedge \omega$ with respect to the Hodge inner product on differential forms.  Define the Laplacians
\begin{eqnarray*}
	\Delta'&=&D'_hD_h'^*+(D'_h)^*D'_h\\
	\Delta''&=&D''(D'')^*+(D'')^*D''
\end{eqnarray*}
 A computation can easily derive the  following equality.
\begin{lem}[Bochner-Kodaira-Nakano identity for Higgs bundles]\label{lem:Bochner}
Let $(E,\theta)$ be a Higgs bundle endowed with a smooth Hermitian metric $h$, which satisfies \cref{assum}. Then
	\begin{align}
	\Delta''=\Delta'+[\sqrt{-1}F(h),\Lambda_\omega]
	\end{align}
\end{lem}  
\begin{proof}
	By \eqref{eq:kahler2}, one has
	$$
	\Delta''=D''(D'')^*+(D'')^*D''=-\sqrt{-1}[D'',[\Lambda_\omega,D_h']].
	$$
	By the Jacobi identity, one has
	\begin{align*}
	\Delta''&= \sqrt{-1}[D_h',[\Lambda_\omega,D'']]-\sqrt{-1}[\Lambda_{\omega},[D'_h,D'']]\\
	&\stackrel{\eqref{eq:kahler}}{=}[D_h',(D_h')^*]+[\sqrt{-1}[D'_h,D''],\Lambda_{\omega}]\\
	&\stackrel{\eqref{eq:assumption}}{=}\Delta'+[\sqrt{-1}F(h),\Lambda_\omega],
	\end{align*}
	which is the desired equality.
\end{proof}
\subsection{Notions of  positivity}
Let us recall the definitions of Nakano positivity and Griffiths negativity for vector bundles in \cite[Chapter \rom{7} \S 6]{Dembook}. 
Let $E$ be a holomorphic vector bundle  endowed with a smooth Hermitian metric $h$. For any $x\in X$, let $e_1,\ldots,e_r$ be a frame of $E$ at $x$, and let $e^1,\ldots,e^r$  be its dual in  $E^*$. Let $z_1,\ldots,z_n$ be a local coordinate centered at $x$. Its curvature tensor is written as
$$
 R(h) =R_{j\bar{k}\alpha}^\beta dz_j\wedge d\bar{z}_k\otimes e^\alpha\otimes e_\beta
$$
Set  $R_{j\bar{k}\alpha\bar{\beta}}:=h_{\gamma\bar{\beta}}R_{j\bar{k}\alpha}^\gamma$, where $h_{\gamma\bar{\beta}}=h(e_\gamma,e_\beta)$.  $(E,h)$ is called \emph{Nakano semi-positive} at $x$ if 
$$
\sum_{ j,k,\alpha,\beta}R_{j\bar{k}\alpha\bar{\beta}}u^{j\alpha} \overline{u^{k\beta}} \geq 0
$$ 
for any $u=\sum_{j,\alpha}u^{j\alpha}\frac{\partial}{\partial z_j}\otimes e_\alpha\in (T_{X}^{1,0}\otimes E)_x$. $(E,h)$ is called \emph{Griffiths semi-negative} at $x$ if 
$$
\sum_{ j,k,\alpha,\beta}R_{j\bar{k}\alpha\bar{\beta}}\xi^j\zeta^\alpha \overline{\xi^k} \overline{\zeta^{\beta}} \leq 0
$$ 
for any $\xi=\sum_{j}\xi^j\frac{\partial}{\partial z_j}\in  T_{X,x}^{1,0}$ and any $\zeta=\sum_{\alpha}\zeta^\alpha e_\alpha\in E_x$.  

We write 
$$R(h)\geq_{\nak} \lambda (\omega\otimes \vvmathbb{1}_E) \quad \mbox{ for } \lmd\in \mathbb{R}$$
if
$$
\sum_{ j,k,\alpha,\beta}(R_{j\bar{k}\alpha\bar{\beta}}-\lambda\omega_{j\bar{k}}h_{\alpha\bar{\beta}})(x)u^{j\alpha}\overline{u^{k\beta}}   \geq 0
$$
for    any $x\in X$ and any $u=\sum_{j,\alpha}u^{j\alpha}\frac{\partial}{\partial z_j}\otimes e_\alpha\in (T_{X}^{1,0}\otimes E)_x$. 
We denote by
$$R(h)\leq_{\gri} \lambda (\omega\otimes \vvmathbb{1}_E) $$
if
$$
\sum_{ j,k,\alpha,\beta}(R_{j\bar{k}\alpha\bar{\beta}}-\lambda\omega_{j\bar{k}}h_{\alpha\bar{\beta}})(x)\xi^j\zeta^\alpha \overline{\xi^k} \overline{\zeta^{\beta}} \leq 0   
$$
for any     $x\in X$,  any  $\xi=\sum_{j}\xi^j\frac{\partial}{\partial z_j}\in  T_{X,x}^{1,0}$ and any $\zeta=\sum_{\alpha}\zeta^\alpha e_\alpha\in E_x$.  Note that Nakano semi-positivity (resp. semi-negativity) implies Griffiths semi-positivity (resp. semi-negativity).

\begin{lem}\label{lem:acceptable}
	Let $(E,h)$ be a hermitian vector bundle on a K\"ahler manifold $(X,\omega)$. If there is a positive constant $C$ so that $|R(h)(x)|_{h,\omega}\leq C$ for any $x\in X$, then 
	$$
C \omega\otimes \vvmathbb{1}_{E}\geq_{\nak} 	R(h)\geq_{\nak} -C \omega\otimes \vvmathbb{1}_{E}
	$$
\end{lem}
\begin{proof}
For any $x\in X$, let $z_1,\ldots,z_n$ be a local coordinate centered at $x$ so that
 $$
\omega_x=\sqrt{-1} \sum_{\ell=1}^{n}dz_\ell \wedge d\bar{z}_\ell
$$
Let   $e_1,\ldots,e_r$ be a local holomorphic frame of $E$ which is orthonormal at $x$. Write 
$$
R(h)=R_{j\bar{k}\alpha}^\beta dz_j\wedge d\bar{z}_k\otimes e^\alpha\otimes e_\beta.
$$
Then $R_{j\bar{k}\alpha\bar{\beta}}(x)=R_{j\bar{k}\alpha}^\beta(x)$, and we have
$$
\sum_{ j,k,\alpha,\beta}|R_{j\bar{k}\alpha\bar{\beta}}(x)|^2=|R(h)(x)|^2_{h,\omega}\leq C^2.
$$
Hence for any $u=\sum_{j,\alpha}u^{j\alpha}\frac{\partial}{\partial z_j}\otimes e_\alpha\in (T_{X}^{1,0}\otimes E)_x$, one has
\begin{align*}
|\sum_{ j,k,\alpha,\beta}R_{j\bar{k}\alpha\bar{\beta}}(x)u^{j\alpha}\overline{u^{k\beta}}|^2&\leq \sum_{j,\alpha}|\sum_{ k,\beta}R_{j\bar{k}\alpha\bar{\beta}}(x)u^{j\alpha}\overline{u^{k\beta}}|^2 \\
 &\leq \sum_{j,\alpha}(\sum_{  k, \beta}|R_{j\bar{k}\alpha\bar{\beta}}(x)u^{j\alpha}|^2)\cdot (\sum_{ k,\beta}|\overline{u^{k\beta}}|^2)\\
 &=|u|_{h,\omega}^2\cdot \sum_{j,\alpha}(\sum_{k,\beta}|R_{j\bar{k}\alpha\bar{\beta}}(x)u^{j\alpha}|^2)\\
 &\leq |u|_{h,\omega}^2\cdot \sum_{j,\alpha}(\sum_{k,\beta}|R_{j\bar{k}\alpha\bar{\beta}}(x)|^2)(\sum_{j,\alpha}|u^{j\alpha}|^2)\\
 &\leq |u|_{h,\omega}^4\cdot\sum_{ j,k,\alpha,\beta} |R_{j\bar{k}\alpha\bar{\beta}}(x)|^2
 \leq |u|_{h,\omega}^4\cdot C^2.
\end{align*}
Hence one has
$$
-C|u|_{h,\omega}^2\leq \sum_{ j,k,\alpha,\beta}R_{j\bar{k}\alpha\bar{\beta}}(x)u^{j\alpha}\overline{u^{k\beta}}\leq C|u|_{h,\omega}^2
$$
The lemma is proved.
\end{proof}
The following easy fact will be useful in this paper.
\begin{lem}\label{lem:tensor}
Let $(E_1,h_1)$ and $(E_2,h_2)$ are two hermitian vector bundles over a K\"ahler manifold $(X,\omega)$ such that $|R(h_1)(x)|_{h_1,\omega}\leq C_1$ and $|R(h_2)(x)|_{h_2,\omega}\leq C_2$ for all $x\in X$. Then for the hermitian vector bundle $(E_1\otimes E_2,h_1h_2)$,  one has
$$
|R(h_1h_2)(x)|_{h_1h_2,\omega}\leq \sqrt{2r_2C^2_1+2r_1C^2_2}
$$
for all $x\in X$. Here $r_i:=\rank E_i$.
\end{lem}
\section{$L^2$-method for Higgs bundles}

\subsection{A quick tour for the simplest case}
In this subsection, we assume that $(E,\theta,h)$ is a harmonic bundle over  a projective manifold $X$. We will show how to apply Bochner technique to give a simple and quick proof of \cref{main} in the case that $D=\varnothing$ and $L$ is ample. The main goal of this subsection is to show the general strategy and we will discuss  how to generalize these ideas to prove \cref{main}.

For a Higgs bundle $(E,\theta)$ over a projective manifold $X$ of dimension $n$, one has the following \emph{holomorphic Dolbeault complex}
\begin{align}
{\rm Dol}(E,\theta):=E\xrightarrow{ \theta}E\otimes \Omega^1_X\xrightarrow{ \theta} \cdots \xrightarrow{ \theta} E\otimes \Omega^n_X
\end{align}
By Simpson \cite{Sim92}, the complex of $\mathscr{C}^{\infty}$ sections of $E$
\begin{align}
\sA^0(E)\xrightarrow{D''}\sA^1(E)\xrightarrow{D''}\cdots \xrightarrow{D''}\sA^{2n}(E)
\end{align}
gives a fine resolution of  the above holomorphic Dolbeault complex. Indeed, it can be proven easily from the Dolbeault lemma. Here $\sA^m(E)$ is the sheaf of germs of smooth $m$-forms with value in $E$.  Hence the cohomology of the complex of its global sections $\big(A^\bullet(E),D'' \big)$ computes the hypercohomology $\mathbb{H}^\bullet\big(X,{\rm Dol}(E,\theta)\big)$.

Suppose now $(\tilde{E},\tilde{\theta})$ is a stable Higgs bundle with vanishing Chern classes. By the Simpson correspondence (see \cite{Sim92}), there is a unique (up to a constant rescaling) hermitian metric $\tilde{h}$ over $\tilde{E}$ so that the curvature $F(\tilde{E},\tilde{h})= 0$. For the ample line bundle $L$ on $X$, we choose a smooth Hermitian metric $h_L$ so that its curvature tensor $\sqrt{-1}R(L,h_L)$ is a K\"ahler form $\omega$. 

Let us define a new Higgs bundle $(E,\theta):=(\tilde{E}\otimes L,\tilde{\theta}\otimes \vvmathbb{1})$. We introduce a hermitian metric  $h$   on $E$ defined by $h:=\tilde{h}\otimes h_L$.  One can easily check that $(E,\theta,h)$ satisfies \Cref{assum} and the curvature 
	\begin{align}\label{curvature}
	\sn F(E,h):=\sn R(E,h)+\sn [\theta,\theta^*]= \sn R(L,h_L)\otimes \vvmathbb{1}_E=  \omega\otimes \vvmathbb{1}_E.
	\end{align}  
By the Hodge theory, for each $i\in\mathbb{Z}_{\geqslant 0}$, we know that the space of harmonic forms
$$
\mathscr{H}^i:=\{\alpha\in A^i(E)\mid \Delta''\alpha=0 \}
$$
is isomorphic to the cohomology $H^i\big(A^\bullet(E),D'' \big)\simeq \mathbb{H}^i\big(X,{\rm Dol}(E,\theta)\big)$.

\begin{thm}[\cref{main} in the case that $D=\varnothing$ and $L$ is ample]
	With the notations in this subsection,	$\mathbb{H}^i\big(X,{\rm Dol}(\tilde{E},\tilde{\theta})\otimes L\big)=\mathbb{H}^i\big(X,{\rm Dol}( {E}, {\theta}) \big)=0$ for $i>n$.
\end{thm}
\begin{proof}
	Note that ${\rm Dol}(E,\theta)={\rm Dol}(\tilde{E},\tilde{\theta})\otimes L$. It suffices to prove that $\sH^i=0$  for $i>n$. We will prove by contradiction. Let us take the K\"ahler form $\omega:=\sqrt{-1}R(L,h_L)$. Assume that there exists a non-zero $\alpha\in \sH^i$. Then by \cref{lem:Bochner}, one has
	\begin{align}
	0=\Delta''\alpha=\Delta'\alpha+[\sqrt{-1}F(E,h),\Lambda_\omega]\alpha
	\end{align}
	An integration by parts yields
	$$
	\langle \Delta'\alpha,\alpha\rangle_{h,\omega}=\lVert D_h'\alpha\rVert_{h,\omega}^2+\lVert (D_h')^*\alpha\rVert_{h,\omega}^2\geqslant 0.
	$$
	Hence
	\begin{align*}
	0&\geqslant\int_{X}\langle[\sqrt{-1}F(E,h),\Lambda_\omega]\alpha,\alpha\rangle_{h,\omega}d\vol_\omega\\
	&\stackrel{\eqref{curvature}}{=}\int_{X}\langle[\omega\otimes \vvmathbb{1},\Lambda_\omega]\alpha,\alpha\rangle_{h,\omega}d\vol_\omega\\
	&=\int_{X}(i-n)\lvert\alpha\rvert_{h,\omega}d\vol_\omega>0
	\end{align*}
	for $i>n$. Here $d\vol_\omega:=\frac{\omega^n}{n!}$ denotes the volume form of $(X,\omega)$.
	Hence the contradiction.
\end{proof}

Hence the above proof inspires us that, to prove \cref{main} in full generality, we shall find a \enquote*{proper} complex of fine sheaves which is quasi-isomorphic to ${\rm Dol}(E,\theta)$, so that its cohomology of global sections can be computed explicitly. Inspired by the work \cite{Zuc79,DPS01,HLWY16}, we will consider the $L^2$-complex  as the candidate for this complex of fine sheaves.  However, instead of  solving $\db$-equation for vector bundles to prove the vanishing theorem, we shall consider $L^2$-estimate and solvability criteria of $(\db_E+\theta)$-equations for Higgs bundles $(E,\theta)$. This is the main content of next subsection. 
\subsection{H\"ormander $L^2$-estimate  for Higgs bundles}\label{sec:H-est-for-Higgs}
Solvability criteria for $\db$
-equations on complex manifolds are often described as cohomology vanishing theorems. 
It is essentially based on the abstract theory of functional analysis. Since the K\"ahler identities \eqref{eq:kahler} and \eqref{eq:kahler2} hold  for Higgs bundles, it inspires us that the following principle should hold.
\begin{pri}
The package of $L^2$-estimate by H\"ormander,  Andreotti-Venssetti, Bombieri, Skoda, Demailly et. al.  should hold without modification for Higgs bundles, provided that the $D''=\db+\theta$ is used in place of $\db$ and that $m$-forms are  used instead of $(p,q)$-forms.
\end{pri}

{In this subsection, we work for a very general setup. Let $(E,\db_E,\theta,h)$ be a Higgs bundle together with a Hermitian metric $h$ over a complete   K\"ahler manifold $(M, \omega_M)$ (not necessarily compact). Denote again $D''\coloneqq \bar{\partial}_E+\theta$. Under a certain curvature condition of $(E,\db_E,\theta,h)$, one can solve the $D''$-equation in the same vein as  \cite[Chapter \rom{8}, Theorem 4.5]{Dembook}.  We follow the standard method  of $L^2$ estimate as that in  \cite[Chapter \rom{8}]{Dembook}, and we   provide full details for completeness sake. The results in this section will be applied more specifically to modified complete K\"ahler metrics over complements of simple normal crossing divisors on compact K\"ahler manifolds in \cref{sec:fine}.}

\medskip

Let us denote by $A^{m}(M,E)$ (resp. $A^{p,q}(M,E)$) the set of smooth $E$-valued $m$-forms (resp. $(p,q)$-forms) on $M$, and denote by $A^{m}_0(M,E)$ (resp. $A^{p,q}_0(M,E)$) the set of smooth $E$-valued $m$-forms (resp. $(p,q)$-forms)  on $M$ with compact support over the K\"ahler manifold $(M,\omega_M)$. The pointwise length of $u\in A^m(M,E)$ with respect to the fiber metric induced by $h$ and $\omega_M$, is denoted by $|u|_{h,\omega_M}$. The pointwise inner product of $u$ and $v$ is denoted by 
$\langle u,v\rangle_{h,\omega_M}$, or simply by $\langle u,v\rangle$. Then the $L^2$-norm of $u$,   denoted by $\lVert u\rVert_{h,\omega_M}$, or simply by $\lVert u\rVert$, is defined as the square root of the integral
$$
\lVert u\rVert^2:=\int_{M}|u|^2_{h,\omega_M}d\vol_{\omega_M} 
$$
where $d\vol_{\omega_M}:=\frac{\omega_M^n}{n!}$, which is finite if $u\in A^m_0(M,E)$. The inner product of $u$ and $v$ associated to this norm is defined by
$$
\llangle u,v\rrangle_{h,\omega_M}:= \int_{M}\langle u,v\rangle_{h,\omega_M}d\vol_{\omega_M} 
$$
which is simply denoted by $\llangle u,v\rrangle$. 
 Note that the Hodge decomposition $A^m_0(M,E)=\oplus_{p+q=m}A^{p,q}_0(M,E)$ is orthogonal with respect to this inner product $\llangle\bullet,\bullet\rrangle$. 
 
We shall denote by $L^{m}_{(2), \loc}(M,E)$ (resp.  $L^{p,q}_{(2), \loc}(M,E)$) $E$-valued $m$-forms (resp. $(p,q)$-forms) with locally integrable coefficients.  One has a natural decomposition
$$
L^{m}_{(2), \loc}(M,E)=\oplus_{p+q=m}  L^{p,q}_{(2), \loc}(M,E)
$$
Moreover, the operators $D''$ (and $D_h'$, $\db_E$ respectively) act on  $L^{m}_{(2), \loc}(M,E)$ in the sense of distribution, or precisely speaking, \emph{$E$-valued currents}. Note that the definition of those objects is independent of the choice of the metrics $\omega_M$ and $h$.   A section $s\in L^{m}_{(2), \loc}(M,E)$ is said to be in the domain of definition of $D''$, denoted by $\Dom_{\loc} D''$, if $D''s\in L^{m+1}_{(2), \loc}(M,E)$.
 
 Let $L^{m}_{(2)}(M,E)_{h,\omega_M}$ (resp.  $L^{p,q}_{(2)}(M,E)_{h,\omega_M}$) be the completion of the pre-Hilbert space $A^{m}_0(M,E)$ (resp.  $A^{p,q}_{0}(M,E)$)  with respect to the above inner product $\llangle\bullet,\bullet\rrangle$.  We simply write $L^{m}_{(2)}(M,E)$ (resp.  $L^{p,q}_{(2)}(M,E)$)  if no confusion happens. By the Lebesgue's theory of integration, $L^{m}_{(2)}(M,E)$ (resp.  $L^{p,q}_{(2)}(M,E)$) is a subset of $L^{m}_{(2), \loc}(M,E)$ (resp.  $L^{p,q}_{(2), \loc}(M,E)$). 
 The natural decomposition
 $$
 L^{m}_{(2)}(M,E)=\oplus_{p+q=m}  L_{(2)}^{p,q} (M,E)
 $$
 is orthogonal with respect to the inner product $\llangle\bullet,\bullet\rrangle$.
 
  Hence $D''$ (and $D_h'$, $\db_E$ respectively) act on them respectively, and these operators are unbounded, densely defined linear operators 
$$
L^m_{(2)}(M,E)\to L^{m+1}_{(2)}(M,E). 
$$  
The  domain  of definition of $D''$ denoted by $\Dom D''$ are defined by
$$
\{u\in L^m_{(2)}(M,E) \mid D'' u\in L^{m+1}_{(2)}(M,E)  \},
$$
for which one has $\Dom D''\subset \Dom_{\loc}D''$. Note that $\Dom D''$ depends on the choice of the metric $\omega_M$ and $h$, up to mutual boundedness. Namely, if $\tilde{\omega}_M\sim \omega_M$ and $\tilde{h}\sim h$,   $\Dom D''$ remains the same in terms of the new metrics $\tilde{\omega}_M$ and $\tilde{h}$. 

By the argument in \cite[Chapter \rom{8}, Theorem 1.1]{Dembook}, this extended operator $D''$ (the so-called \emph{weak extension} in the literature) is closed, namely its graph is  closed.   $\Dom D_h'$ is defined in exactly the same manner.

The following   result in \cite[Chapter \rom{8}, Theorem 3.2.(a)]{Dembook}  is crucial in applying the $L^2$-estimate. Roughly speaking, it gives a condition when the weak extension of $D''$ is the strong one, in terms of the graph norm, and it enables us to apply the integration by parts for $L^2$-sections as in \cref{lem:integration}.
\begin{thm} \label{thm:norm}
	Let $(M,\omega_M)$ be a complete K\"ahler manifold and $(E,\db_E,\theta,h)$ is a Higgs bundle on $M$ satisfying \Cref{assum}. Then $A^m_{0}(M,E)$ is dense in $\Dom D''$, $\Dom D''^*$ and $\Dom D''\cap \Dom D''^*$ respectively for the graph norm
		$$
		u\mapsto \lVert u\rVert +\lVert D''u\rVert, \quad 	u\mapsto \lVert u\rVert +\lVert (D'')^*u\rVert, \quad u\mapsto \lVert u\rVert +\lVert D''u\rVert+\lVert  (D'')^*u\rVert.
		$$  
\end{thm} 

We recall the following lemma of functional analysis by Von  Neumann and H\"omander (See e.g., \cite[Chapter \rom{8}, Section 1]{Dembook}), which is crucial in obtaining the $L^2$-estimate for Higgs bundles. First we recall the following notation of the adjoint operator $T^*$ and $\Dom T^*$: $y\in \Dom T^*$ if the linear form
	$$\Dom T\ni x\mapsto \llangle Tx,y\rrangle_2$$ 
	is  bounded  in $\sH_1$-norm. Since $\Dom T$ is dense, there exists for every $y$ in $\Dom T^*$ a
	unique element $T^*y$ in 
	$\sH_1$ such that $\llangle x,T^*y\rrangle_1=\llangle Tx,y\rrangle_2 $ for all $x\in \Dom T$.
 
\begin{lem}\label{neumann:perp}
If  $T:\sH_1\to \sH_2$ is a closed and densely defined operator, then its adjoint $T^*$ is also closed and densely defined and $(T^*)^*=T$. Furthermore, we have the relation $\ker T^*=(\Ima T)^{\perp}$ and its dual $(\ker T)^{\perp} = \overline{\Ima T^*}$. In particular, $\ker T\oplus \overline{\Ima T^*}=\sH_1$.
\end{lem}



 Note that     $A_{m}:=[\sqrt{-1}F(h),\Lambda_{\omega_M}]$ acts on $\wedge^{m}T_M^*\otimes E$ as a    \emph{hermitian operator}. As $A_m$ is smooth, for any $u\in L^m_{(2),\loc}(M,E)$, $A_m(u)\in L^m_{(2),\loc}(M,E)$. If $A_m$ is semi-positively definite, $A_m^{\frac{1}{2}}$ exists as a densely defined   hermitian operator from $L^m_{(2) }(M,E)$ to itself. The following result is exactly the same vein as the Kodaira-Nakano inequality (see \cite[lemme 4.4]{Dem82})
\begin{lem}\label{lem:integration}
Let $(M,\omega_M)$ be a complete K\"ahler manifold  and $(E,\db_E,\theta,h)$ is a Higgs bundle on $M$ satisfying \Cref{assum}.  Assume that $A_m$ is semi-positively definite. Then for every $u\in \Dom D''\cap \Dom D''^*$, one has
\begin{align}\label{eq:Kod}
\lVert D''u\rVert^2+\lVert D''^*u\rVert^2\geq \llangle A_mu,u\rrangle :=\int_M \langle A_mu,u\rangle_{h,\omega_M} d\vol_{\omega_M}
\end{align}
\end{lem}

\begin{proof}  
Since $(M,\omega_M)$ is complete, by the proof of \cite[Chapter \rom{8}, Theorem 3.2.(a)]{Dembook}, there   exists an exhaustive sequence $\{K_{\nu}\}_{\nu\in \mathbb{N}}$ of compact subsets of $M$ and functions
$\rho_\nu$  such that
 $\rho_\nu=1$ in a neighborhood of $K_\nu$, $\text{Supp}(\rho_\nu)\subset K_{\nu+1}$, $0\leq \rho_\nu\leq 1$, and $|d\rho_\nu|_{\omega_M}\leq 2^{-\nu}$. 
 One can show that $\rho_\nu u\to u$ in the graph norm $u\mapsto \lVert u\rVert +\lVert D''u\rVert+\lVert  D''^*u\rVert 
  $. Since $A_m$ is supposed to be semi-positively definite, hence  by the monotone convergence theorem
  $$
  \lim\limits_{\nu\to +\infty}\int_M \langle A_m(\rho_\nu u),\rho_\nu u\rangle_{h,\omega_M} d\vol_{\omega_M}=  \int_M \langle A_m( u),  u\rangle_{h,\omega_M} d\vol_{\omega_M},
  $$
  which might be $+\infty$ in general. Hence it suffices to prove \eqref{eq:Kod} under the assumption that $u$ has compact support.

By the convolution arguments in \cite[Chapter \rom{8}, Theorem 3.2.(a)]{Dembook},  there exists $u_\ell \in A_0^m(M,E)$ so that $u_\ell$ tends  to $ u$ as $\ell\to \infty$ with respect to the graph norm $\lVert u\rVert +\lVert D''u\rVert+\lVert  D''^*u\rVert$, and there is a uniform compact set $K$ so that $\text{Supp}(u_\ell)\subset K$ for all $\ell$. By \cref{lem:Bochner}, one has
\begin{align*}
\llangle \Delta''u_\ell,u_\ell\rrangle= \llangle \Delta'u_\ell,u_\ell\rrangle+\llangle A_mu_\ell,u_\ell\rrangle
\end{align*}
As $u_\ell $ has compact support, one applies integration by parts to obtain
$$
\llangle \Delta''u_\ell,u_\ell\rrangle=\lVert D''u_\ell\rVert^2+\lVert D''^*u_\ell\rVert^2
$$ 
and 
$$
\llangle \Delta'u_\ell,u_\ell\rrangle=\lVert D_h'u_\ell\rVert^2+\lVert D'^*u_\ell\rVert^2\geq 0
$$
which gives rise to 
$$\lVert D''u_\ell\rVert^2+\lVert D''^*u_\ell\rVert^2\geq \llangle A_mu_\ell,u_\ell\rrangle$$
 \eqref{eq:Kod} follows from the above inequality when $\ell$ tends to infinity. The lemma is proved.
\end{proof}

\begin{rem}\label{rem:converse}
Suppose that $A_m$ is a semi-positively definite hermitian operator on $\wedge^{m}T_M^*\otimes E$. 
 For some $v\in L^m_{(2)}(M,E)$, assume that for almost all $x\in M$, there exists   a measurable and integrable non-negative function  $\alpha(x)$ so that
$$
| \langle v,f\rangle_{h,\omega_M}|^2  \leq	\alpha(x) \langle f,A_m(x)f\rangle_{h,\omega_M} 
$$ 
for any $f\in A^m_0(M,E)_x$, then the minimum of $\alpha(x)$  
is $$|A_m^{-\frac{1}{2}}(x)v|^2_{h,\omega_M}=\langle A_m(x)^{-1}v,v\rangle_{h,\omega_M}$$ if the operator $A_m(x)$ is invertible. Hence we shall always formally write it in this way even when
$A_m(x)$ is no longer invertible, 
following \cite[Chapter \rom{8}, \S4]{Dembook}.  
\end{rem}

Now we are able to state our main result on $L^2$-estimate for Higgs bundles.
 
%
\begin{thm}[Solving $D''$-equation for Higgs bundle]\label{thm:L2}
Let $(M,\omega_M)$ be a complete K\"ahler manifold, and $(E,\db_E,\theta,h)$ be a Higgs bundle on $M$ satisfying   \Cref{assum}. Assume that $A_m$ is semi-positively definite on $\wedge^mT_M^*\otimes E$ at every $x\in M$. Then for any $v\in L^m_{(2)}(M,E)$   such that $D''v=0$ and  
	$$
\int_M \langle A_m^{-1}v,v\rangle d\vol_{\omega_M}<+\infty,
	$$
	there exists $u\in   L^{m-1}_{(2)}(M,E)$ so that $D''u=v$ and
	$$
	\lVert u\rVert^2\leqslant \int_M \langle A_m^{-1}v,v\rangle d\vol_{\omega_M}.
	$$
\end{thm}
\begin{proof}
Consider now two closed and densely defined operators
	$$
	\sH_1=L^{m-1}_{(2)}(M,E)\xrightarrow{T=D''}\sH_2=L^{m}_{(2)}(M,E)\xrightarrow{S=D''} \sH_3=L^{m+1}_{(2)}(M,E).$$
For any $f\in \Dom S\cap \Dom T^*$, one has
\begin{align*} 
|\llangle f,v\rrangle|^2=|\int_M \langle f,v\rangle d\vol_{\omega_M}|^2\leq  |\int_M  \langle A_m^{-1}v,v\rangle^{\frac{1}{2}}  \cdot  \langle A_mf,f\rangle^{\frac{1}{2}} d\vol_{\omega_M}|^2\\
\leq \int_M \langle A_m^{-1}v,v\rangle d\vol_{\omega_M}\cdot \int_M \langle A_mf,f\rangle d\vol_{\omega_M}
\end{align*}
by Cauchy-Schwarz inequality. By \eqref{eq:Kod} one has
\begin{align}\label{eq:norm}
|\llangle f,v\rrangle|^2\leq C ( \lVert Sf\rVert^2+\lVert T^*f\rVert^2), 
\end{align}
where $C:=\int_M \langle A_m^{-1}v,v\rangle d\vol_{\omega_M}>0$.

Note that $T^*\circ S^*=0$ by $S\circ T=0$.  For any $f\in \Dom T^*$, there is an orthogonal decomposition   $f=f_1+f_2$, where $f_1\in \ker S$ and $f_2\in (\ker S)^{\perp}=\overline{\Ima S^*}\subset \ker T^*$ by \cref{neumann:perp}. Since $v\in \ker S$, by \eqref{eq:norm} we then have
$$
|\llangle f,v\rrangle|^2=|\llangle f_1,v\rrangle|^2\leqslant C ( \lVert Sf_1\rVert^2+\lVert T^*f_1\rVert^2)=C\lVert T^*f_1\rVert^2=C\lVert T^*f\rVert^2.
$$
By Hahn-Banach Theorem, we conclude that there is $u\in \Dom T$ so that $Tu=v$ with $\lVert u\rVert_2\leq C^{1/2}$. The theorem is proved. 
\end{proof}

A direct consequence is the following result which can be seen as a Higgs bundle version of Girbau vanishing theorem (see \cite[Chapter \rom{7}, Theorem 4.2]{Dembook}) in the log setting \cite[Theorem 4.1]{HLWY16}.  
\begin{cor}\label{cor:AN}
	Let $(M,\omega_M)$ be a complete K\"ahler manifold, and $(\tilde{E},\tilde{\theta},\tilde{h})$ be any harmonic bundle on $M$.  Let $L$ be  a   line bundle   on $M$ equipped with a Hermitian metric $h_L$.   Assume that  
\begin{align}\label{eq:critical}
	\langle [\sn R(h_L), \Lambda_{\omega_M}]f, f\rangle_{h_L, \omega_M}\geq \ep \lvert f\rvert_{h_L, \omega_M}^2
\end{align} 
for any $x\in M$ and $f\in (\Lambda^{p,q}T_M^*\otimes L)_x$ with $p+q=m$.  Set $(E,\theta,h):=(\tilde{E}\otimes L,\tilde{\theta}\otimes \vvmathbb{1}_L,\tilde{h}h_L)$. 
Then for any $v\in L^m_{(2)}(M,E)$  such that $D''v=0$, 	there exists $u\in   L^{m-1}_{(2)}(M,E)$ so that $D''u=v$ and
$$
\lVert u\rVert^2\leqslant \frac{\lVert v\rVert^2}{\ep}.
$$
\end{cor}
\begin{proof}
	Note that since $(\tilde{E},\tilde{\theta},\tilde{h})$ is a harmonic bundle, both $(\tilde{E},\tilde{\theta},\tilde{h})$ and $(E,\theta,h)$ satisfy \Cref{assum}. Hence
	\begin{align}\nonumber
		\sn F(h)&=\sn \Big(R(h)+[\theta,\theta_h^*] \Big)\\\nonumber
		&=\sn R(\tilde{h})\otimes \vvmathbb{1}_L+ \sn R(h_L)\otimes \vvmathbb{1}_{E}+[\tilde{\theta}\otimes \vvmathbb{1}_L, \tilde{\theta}_{\tilde{h}}^*\otimes \vvmathbb{1}_L]\\\nonumber
		&=\sn F(\tilde{h})\otimes \vvmathbb{1}_L+  \sn R(h_L)\otimes \vvmathbb{1}_{E}\\\label{eq:formula}
		&= \sn R(h_L)\otimes \vvmathbb{1}_{E},
	\end{align}
	where the last equality follows from that  $F(\tilde{h})=0$ since $(\tilde{E},\tilde{\theta},\tilde{h})$   is a harmonic bundle. 
	In this case, it is easy to see that for any $f\in (\Lambda^{m}T_M^*\otimes E)_x$,   decomposing  $f=\sum_{p+q=m}f^{p,q}$ with $f^{p,q}$ its $(p,q)$-component,  one has
	$$
	\langle A_mf,f\rangle_{h,\omega_M}=	\sum_{p+q=m}\langle [\sn R(h_L), \Lambda_{\omega_M}]\otimes \vvmathbb{1}_E(f^{p,q}), f^{p,q}\rangle_{h_L,\omega_M}\geq  \sum_{p+q=m} \ep\lvert f^{p,q}\rvert^2_{h,\omega_M}= \ep\lvert f\rvert^2_{h,\omega_M}.
	$$
	Hence $	\langle A_m^{-1}f,f\rangle_{h,\omega_M}\leq  \ep^{-1}\lvert f\rvert^2_{h,\omega_M}.$
 Applying \cref{thm:L2},  we conclude that there is $u\in  L^{m-1}_{(2)}(M,E) $ so that $D''u=v$ and 
$$
\lVert u\rVert^2\leqslant \int_M	\langle A^{-1}_mv,v\rangle_{h,\omega_M}d\vol_{\omega_M}\leq \frac{\lVert v\rVert^2}{\ep}.
$$
\end{proof}

\section{Vanishing theorem for tame   harmonic bundles}\label{sec:main}
\subsection{Parabolic Higgs bundle}\label{sec:parahiggs}

In this section, we recall the notions of parabolic Higgs bundles. For more details refer to \cite[section 1, 3, 4, 5]{AHL19} and \cite[section 1]{MY92}. Let $X$ be a complex manifold, $D=\sum_{i=1}^{\ell}D_i$ be a reduced simple normal crossing divisor, $U=X-D$ be the complement of $D$ and $j:U\to X$ be the inclusion.

\begin{dfn}\label{dfn:parab-higgs}
	A parabolic sheaf $(E,{}_{ \bm{a}}E)$ on $(X, D)$ is a torsion free
	$\mathcal{O}_U$-module $E$, together with an $\mathbb{R}^l$-indexed
	filtration ${}_{ \bm{a}}E$ ({\em parabolic structure}) by coherent subsheaves of $j_*E$ such that
	\begin{enumerate}[leftmargin=0.7cm]
		\item $\bm{a}\in \mathbb{R}^l$ and ${}_{\bm{a}}E|_U=E$. 
		\item  For $1\leq i\leq l$, ${}_{\bm{a}+\bm{1}_i}E = {}_{\bm{a}}E\otimes \cO_X(D_i)$, where $\bm{1}_i=(0,\ldots, 0, 1, 0, \ldots, 0)$ with $1$ in the $i$-th component.
		\item $_{\bm{a}+\bm{\epsilon}}E = {}_{\bm{a}}E$ for any vector $\bm{\epsilon}=(\epsilon, \ldots, \epsilon)$ with $0<\epsilon\ll 1$.
		\item  The set of {\em weights} \{$\bm{a}$\ |\  $_{\bm{a}}E/_{\bm{a}-\bm{\epsilon}}E\not= 0$  for any vector $\bm{\epsilon}=(\epsilon, \ldots, \epsilon)$ with $0<\epsilon\ll 1$\}  is  discrete in $\mathbb{R}^l$.
	\end{enumerate}
\end{dfn}
A weight
is normalized if it lies in $[0,1)^l$. Denote $_{\bm{0}}E$ by $\diae$, where $\bm{0}=(0, \ldots, 0)$.  Note that the parabolic structure of $(E,{}_{ \bm{a}}E)$ is uniquely determined by the filtration for weights lying in $[0,1)^l$. A {\em parabolic bundle} on $(X,D)$ consists of a
vector bundle $E$ on $X$ with a parabolic structure, such that the filtered subsheaves ${}_{ \bm{a}}E$ are vector bundles.  As pointed out by the referee, by the work of Borne-Vistoli the parabolic structure of a parabolic bundle  is  \emph{locally abelian}, \emph{i.e.} it admits a local frame compatible with the filtration (see e.g. \cite{IS07} and \cite{BV12}). 

\begin{dfn}
	A {\em parabolic Higgs bundle} on $(X,D)$ is a parabolic
	bundle $(E,{}_{ \bm{a}}E,\theta)$ together with $\sO_X$ linear map
	$$\theta:\diae\to \Omega_X^1(\log D)\otimes \diae$$
	such that
	$$\theta\wedge \theta=0$$
	and
	$$\theta(_{\bm{a}}E)\subseteq \Omega_X^1(\log D)\otimes {}_{\bm{a}}E,$$ for $\bm{a}\in [0, 1)^l$.
\end{dfn}

A natural class of parabolic Higgs bundles comes from prolongations of tame harmonic bundles, which is discussed in the following section.

\subsection{Prolongation by an increased order}\label{sec:prolongation}
By a celebrated theorem of Simpson and Mochizuki, there is a natural parabolic Higgs bundle induced by tame harmonic bundle $(E, \theta, h)$. 

We recall some notions in \cite[\S 2.2.1]{Moc07}.  Let $(X, D)$ be the pair in subsection \ref{sec:parahiggs}.  Let $E$ be a holomorphic vector bundle with a $\sC^\infty$ hermitian metric $h$ over $X-D$.

Let $U$ be an open subset of $X$ with an admissible coordinate $(U; z_1, \ldots, z_n)$ with respect to $D$. For any section $\sigma\in \Gamma(U-D,E|_{U-D})$, let $|\sigma|_h$ denote the norm function of $\sigma$ with respect to the metric $h$. We denote $|\sigma|_h\in \cO(\prod_{i=1}^{\ell}|z_i|^{-b_i})$ if there exists a positive number $C$ such that $|\sigma|_h\leq C\cdot\prod_{i=1}^{\ell}|z_i|^{-b_i}$. For any $\bm{b}\in \bR^\ell$, say $-\mbox{ord}(\sigma)\leq \bm{b}$ means the following:
$$
|\sigma|_h=\cO(\prod_{i=1}^{\ell}|z_i|^{-b_i-\varepsilon})
$$
for any real number  $\varepsilon>0$ and $0<|z_i|\ll1$. For any $\bm{b}$, the sheaf ${}_{\bm{b}} E$ is defined as follows: 
\begin{align}\label{eq:prolongation}
\Gamma(U-D, {}_{\bm{b}} E):=\{\sigma\in\Gamma(U-D,E|_{U-D})\mid -\mbox{ord}(\sigma)\leq \bm{b} \}. 
\end{align}
The sheaf ${}_{\bm{b}} E$ is called the prolongment of $E$ by an increasing order $\bm{b}$. In particular, we use the notation ${}^\diamond E$ in the case $\bm{b}=(0,\ldots,0)$.

According to  Simpson \cite[Theorem 2]{Sim90} and Mochizuki \cite[Theorem 8.58]{Moc07}, the above prolongation gives a parabolic Higgs bundles, especially $\theta$ preserves the filtration.
\begin{thm}[Simpson, Mochizuki] \label{thm:SM} Let $(X, D)$ be a complex manifold $X$ with a simple normal crossing divisor $D$. If $(E , \theta, h)$ is a tame harmonic bundle on $X-D$, then the corresponding filtration $_{\bm{b}}E$ according to the increasing order in the prolongment of $E$ defines a parabolic bundle $(E, {}_{\bm{b}}E, \theta)$ on $(X,D)$.  	\qed
\end{thm} 

Here we also recall the following definition in \cite[Definition 2.7]{Moc07}.
\begin{dfn}[Acceptable bundle]
	Let $(E,\db_E,h)$  be a hermitian vector bundle over $X-D$. We say that $(E,\db_E,h)$ is acceptable at $p\in D$, if the following holds: there is an  admissible coordinate $(U;z_1,\ldots,z_n)$ around $p$, so that the norm  $\lvert R(E,h) \rvert_{h,\omega_P}\leq C$ for some $C>0$. When $(E_,\db_E,h)$ is acceptable at any point $p$ of $D$, it is called acceptable.  
\end{dfn} 

The following deep result by Mochizuki \cite[Proposition 8.18]{Moc07} will play an important role throughout this paper. 
\begin{thm}[Mochizuki]\label{thm:acceptable}
	Let $X$ be a complex manifold and let $D$ be a simple normal crossing divisor on $X$. Assume that $(E,\theta,h)$ is a tame harmonic bundle on $X-D$. Then $(E,h)$ is acceptable. 
\end{thm}
\subsection{Modification of the metric}\label{sec:dolbealt} 
In this subsection, we work with the following modification of acceptable metric defined in \cite[\S 4.5.3]{Moc02}. 
  Let us consider the case $X=\Delta^n$, and $D=\sum_{i=1}^{\ell}D_i$ with $D_i=(z_i=0)$. Let $(E,\db_E,h)$ be an acceptable bundle over $X-D$. For any $\bm{a}\in \mathbb{R}^\ell_{\geq 0}$ and $N\in \mathbb{Z}$, we define
  \begin{align}\label{eq:modify}
  \chi(\bm{a},N):= -\sum_{j=1}^{\ell}a_j\log |z_j|^{2}- N\big(\sum_{j=1}^{\ell}\log   (-\log |z_j|^2)+ \sum_{k=\ell+1}^{n}\log(1-|z_k|^2)\big).
  \end{align}
  Set $h(\bm{a},N):=h\cdot e^{-\chi(\bm{a},N)}$. Then 
  $$
  R(h(\bm{a},N))=R(h)+\hess \chi(\bm{a},N)=R(h)+N\omega_P.
  $$
  
Note that $\Omega_{X^*}=\oplus_{i=1}^{n}L_i$
where $L_i$ is the trivial line bundle defined by $L_i:=p_i^*\Omega_{\Delta^*}$ for $i=1,\ldots,\ell$ and $L_k=p_k^*\Omega_{\Delta}$ for $k=\ell+1,\ldots,n$ where $p_i$ is the projection of $(\Delta^*)^\ell\times\Delta^{n-\ell}$ to its $i$-th factor. For any $p=1,\ldots,n$, set $h_p$ to be the hermitian metric on $\Omega_{X^*}^p$ induced by $\omega_P$. Then  there is a positive constant $C(p,\ell)>0$ depending only on $p$ and $\ell$ so that $|R(h_p)|_{h_p,\omega_P}\leq C(p,\ell)$.  
Set $C_0:=\sup_{p=0,\ldots,n;\ell=1,\ldots,n}C(p,\ell)$.

  
 \begin{prop}\label{lem:estimate}
 	Let $(E,\db_E,h)$ be an acceptable bundle over $X-D$, where $X$ is a compact complex manifold and $D$ is a simple normal crossing divisor. Then there is a constant $N_0>0$ so that, for any   $x\in D$, one has an admissible coordinate $(U;z_1,\ldots,z_n)$ around $x$ (which can be made arbitrary small) satisfying the following property:
 	
 	For   vector bundles $\sE_p:= T_{U^*}^p\otimes E$ and $\sF_p:=\Omega_{U^*}^p\otimes E$, which are all equipped with the   $\sC^\infty$-metric $h_{\sE_p}$ and $h_{\sF_p}$ induced by  $h(\bm{a},N)$ and $\omega_P$, one  has the following estimate
\begin{align} \label{eq:estimate 2}
  \sqrt{-1}R(h_{\sE_p})\geqslant_{\nak} \omega_P\otimes \vvmathbb{1}_{\sE_p}; \quad 
 \sqrt{-1}R(h_{\sF_p})\leqslant_{\gri} 2N \omega_P\otimes \vvmathbb{1}_{\sF_p}
 \end{align}
over $U^*$ for any $N\geqslant N_0$. Such $N_0$ does not depend on the choice of $\bm{a}$. 
 \end{prop} 
\begin{proof}
	As $(E,h)$ is assumed to be acceptable,  for any $x\in D$, one can find an admissible coordinate    $(U;z_1,\ldots,z_n;\varphi)$ around $x$ so that $|R(h)|_{h,\omega_P}\leq C$. By the above argument, one has $|R(h_p)|_{h_p,\omega_P}\leq C_0$ for the Hermitian metric $h_p$ on $\Omega^p_{U^*}$. By \cref{lem:tensor}, we conclude that there is a constant $C_1>0$ which depends only on $C_0$ and $C$ so that
	$$
	|R(h^{-1}_ph)|_{h_p^{-1}h,\omega_P}\leq C_1,\quad 	|R(h_{p}h)|_{h_ph,\omega_P}\leq C_1
	$$   
	for any $p=0,\ldots,n$, 
	where $h^{-1}_ph$ is the metric for $\sE_p$ and $h_ph$ is the metric for $\sF_p$.  By   \cref{lem:acceptable}, one has
	$$
	\sqrt{-1}R(h^{-1}_ph)\geq_{\nak}  -C_1\omega_P\otimes\vvmathbb{1}_{\sE_p},\quad 	\sqrt{-1}R(h_{p}h)\leq_{\nak}C_1\omega_P\otimes \vvmathbb{1}_{\sF_p}.
	$$
As $h_{\sE_p}=h^{-1}_ph(\bm{a},N)$  and  $h_{\sF_p}=h_ph(\bm{a},N)$, we then have 
	$$
\sqrt{-1}R(h_{\sE_p})\geq_{\nak}  (N-C_1)\omega_P\otimes\vvmathbb{1}_{\sE_p},\quad 	\sqrt{-1}R(h_{\sF_p})\leq_{\nak}(N+C_1)\omega_P\otimes \vvmathbb{1}_{\sF_p}.
$$
	If we take $N_{x}=C_1+1$, then the desired estimate \eqref{eq:estimate 2}  follows for any  $N\geq N_x$.
	 
	Now we will prove that for points near $x$, the above estimate $N_x$ holds uniformly. As $C_1$ depends only on $C$, one has to prove that there is a constant $C$ so that for any point $z$ near $x$, there is an admissible coordinate with respect to $z$ so that $|R(h)|_{h,\omega_P}\leq  C$. 
	\begin{claim}
		Let $\phi:\Delta\to \Delta^*$ defined by $\phi(t)=\frac{t}{4}+\frac{1}{2}$. Then $$\phi^*\frac{\sqrt{-1}dz\wedge d\bar{z}}{|z|^2(\log |z|^2)^2}=\frac{\sqrt{-1}dt\wedge d\bar{t}}{|\phi(t)|^2(\log |\phi(t)|^2)^2}\leq  C_2\sqrt{-1}dt\wedge d\bar{t}\leq C_2\frac{\sqrt{-1}dt\wedge d\bar{t}}{(1-|t|^2)^2},$$ 
		where $C_2=4 (\log \frac{3}{4})^{-2}$.
	\end{claim}  For any $z\in U$, we first assume that $z_1=\cdots=z_\ell =0$, namely the components of $D$ passing to $z$ are the same as $x$. Take isomorphisms of unit disk $\{\phi_j\in \text{Aut}(\Delta)\}_{j=\ell+1,\ldots,n}$ so that $\phi_j(z_j)=x_j$. Note that $x_1=\cdots=x_\ell=0$. Hence $(\vvmathbb{1}_\Delta,\ldots,\vvmathbb{1}_\Delta,\phi_{\ell+1},\ldots,\phi_n)\circ \varphi:U\to \Delta^n$ gives rise to the admissible coordinate for $z$, and the Poincar\'e metric $\omega_P$ is invariant under this transformation. Hence one can take $N_z=N_x$.
	
	Now we can assume that $z_1=\cdots=z_m=0$, and that any of $\{z_{m+1},\ldots,z_{\ell}\}$ is   not equal to zero, for $m<l$.  We first take automorphisms  $\{\eta_i\}_{i=m+1,\ldots,\ell}\subset\text{Aut}(\Delta^*)$ so that $\eta_i(\frac{1}{2})=z_i$. Set $\phi_i=\eta_i\circ \phi:\Delta\to \Delta^*$ for $i=m+1,\ldots,\ell$. Take isomorphisms of unit disk $\{\phi_j\in \text{Aut}(\Delta)\}_{j=\ell+1,\ldots,n}$ so that $\phi_j(z_j)=x_j$. Then $\varphi^{-1}\circ (\vvmathbb{1}_\Delta,\ldots,\vvmathbb{1}_\Delta,\phi_{m+1},\ldots,\phi_n):\Delta^n\to X$ will give rise to the desired admissible coordinate for such $z$. By the above claim, one has $|R(h)|_{h,\omega_P}\leq C_2C$. Hence the above estimate $N_x$ can be made uniformly in $U$. As $X$ and $D$ is compact, one can cover $D$ by finite such open sets, and the desired $N_0$ in the theorem can be achieved. 	 
	
We now show that these admissible coordinates can be made arbitrarily small. For $0<\ep<1$, set
\begin{align*}
\phi_\ep:\Delta_\ep^n&\stackrel{\sim}{\to } \Delta^n\\
(z_1,\ldots,z_n)&\to (\ep^{-1} z_1,\ldots,\ep^{-1} z_n),
\end{align*}
where $\Delta_\ep=\{z\in \Delta \mid |z|<\ep \}$.
For any  admissible coordinate $(U;z_1,\ldots,z_n;\varphi)$ around $x$ so that $|R(h)|_{h,\omega_P}\leq C$, one  can introduce a new   one $(U(\ep);w_1,\ldots,w_n;\varphi_\ep)$ around $x$   with 
\begin{align*}
\varphi_\ep:U(\ep)&\stackrel{\sim}{\to} \Delta^n \\
x&\to  \phi_\ep\circ \varphi(x).
\end{align*}
When $\ep\ll 1$, this admissible coordinate will be arbitrarily small. Note that  $\phi_\ep^*\omega_P\geq \omega_P|_{\Delta_\ep^n}$.   Hence in  the new admissible coordinate $(U(\ep);w_1,\ldots,w_n;\varphi_\ep)$, one still have $\lvert R(h) \rvert_{h,\omega_P}\leq C$. The constant $N_x$ is thus unchanged. The proposition is proved.
	\end{proof}
This result will be important for us to construct a fine resolution of parabolic Higgs bundles in \cref{sec:fine}.

  \subsection{From $L^2$-integrability  to $\sC^0$-estimate} 
 Note that in order to show the quasi-isomorphism between some complex of sheaves of $L^2$-forms and   \eqref{eq:dol}, one has to deduce some norm estimate of sections   from the  $L^2$-integrability condition.  In the case that $(E,\theta)$ is a line bundle with trivial Higgs field, this has been carried out in \cite[\S 2.4.2]{DPS01} and \cite[Theorem 3.1]{HLWY16}.  This subsection is devoted to show this using \emph{mean value inequality} following \cite[Lemma 7.12]{Moc06}.


 We first recall the following well-known lemma and we provide the proof for completeness sake.
  \begin{lem}\label{lem:psh} Let $(E, h)$ be a Hermitian vector bundle over a complex manifold $X$. Suppose that $R(h)$ is Griffiths semi-negative.  Then for any holomorphic section $s\in H^0(X,E)$, one has
  	$$
  	\hess \log |s|^2_h\geqslant 0.
  	$$
  \end{lem}

\begin{proof}
	Outside the zero locus $(s=0)$, one has
	\begin{align*}
 	\hess \log |s|^2_h&=\sqrt{-1}\frac{\{\partial_h s,\partial_h s\}_h}{|s|^2_h}-\sqrt{-1}\frac{\{\partial_hs,s\}_h\wedge \{s,\partial_hs\}_h}{|s|^4_h}-\frac{\{\sqrt{-1}R(h)s,s\}_h}{|s|^2_h}\\
 	&\geq -\frac{\{\sqrt{-1}R(h)s,s\}_h}{|s|^2_h}\geq 0
	\end{align*}
	where the first inequality is due to Cauchy-Schwarz inequality and the second one follows from the assumption that $R(h)$ is Griffiths semi-negative.  As $\log|s|_h^2$ is locally bounded from above, it is thus a global plurisubharmonic function on $X$.
\end{proof}
 
\begin{prop}\label{prop:C0}
With the same setting as \Cref{lem:estimate},	 for any   $x\in D$,  we take an admissible coordinate $(U;z_1,\ldots,z_n)$ around $x$ and pick $N\geq N_0$ as in \Cref{lem:estimate}. Then 	for any section $s\in H^0(U^*, \Omega_{U^*}^p\otimes E|_{U^*})$,  when $0<r\ll 1$, one has 
	\begin{align}
	|s|_{h,\omega_P}(z)\leq C \lVert s\rVert_{h(\bm{a},N),\omega_P}\cdot (\prod_{i=1}^{\ell}|z_i|^{-a_i-\delta})
	\end{align} 
for any $\delta>0$ and any $z\in U^*(r)$.
\end{prop}
\begin{proof}
	 By \Cref{lem:estimate}, for the hermitian vector bundle $(\Omega^p_{U^*}\otimes E, h_ph(\bm{a},-N))$   one thus has 
	$$
	R(h_ph(\bm{a},-N))=	R(h_ph(\bm{a},N))-2N\omega_P\otimes \vvmathbb{1}_{\Omega^p_{U^*}\otimes E} \leqslant_{\gri} 0
	$$
	over $U^*$ for $N\geqslant N_0$. 
	For any section $s\in H^0(U^*,\Omega^p_{U^*}\otimes E)$, by \cref{lem:psh} one has
	$$
	\hess \log |s(z)|^2_{h(\bm{a},-N),\omega_P}\geq 0,
	$$ 
	where we omit $h_p$ in the subscript for simplicity. 
	For any $z\in U^*(r)$ where $0<r\ll 1$, one has $\log |s(z)|^2_{h(\bm{a},-N),\omega_P}<0$, and 
	\begin{align*}
	\log |s(z)|^2_{h(\bm{a},-N),\omega_P}&\leqslant  \frac{4^n}{\pi^n \prod_{i=1}^{\ell}|z_i|^2}\int_{\Omega_z}  \log |s(w)|^2_{h(\bm{a},-N),\omega_P}d\mbox{vol}_{g}\\
	&\leqslant \log \big(\frac{4^n}{\pi^n \prod_{i=1}^{\ell}|z_i|^2}\cdot \int_{\Omega_z}   |s(w)|^2_{h(\bm{a},-N),\omega_P}d\mbox{vol}_{g}\big)\\
	&\leqslant \log \big(C \int_{\Omega_z}    \frac{1}{  \prod_{i=1}^{\ell}|w_i|^2}|s(w)|^2_{h(\bm{a},-N),\omega_P}d\mbox{vol}_{g}\big)\\
	&\leqslant \log C_1+ \log \int_{\Omega_z}   |s(w)|^2_{h(\bm{a},-N),\omega_P}\cdot |\prod_{i=1}^{\ell}(\log |w_i|^2)^2| \prod_{j=\ell+1}^{n}(1-|w_j|^2)^2 d\mbox{vol}_{\omega_P}\\
	& \leqslant  \log C_1+\log   \int_{\Omega_z}   |s(w)|^2_{h(\bm{a},N),\omega_P}d\mbox{vol}_{\omega_P}\\
	&\leqslant \log C_1+\log \lVert s\rVert_{h(\bm{a},N),\omega_P}^2
	\end{align*}
	where $\Omega_z:=\{w\in U^*\mid |w_i-z_i|\leq \frac{|z_i|}{2} \mbox{ for } i\leq \ell;   |w_i-z_i|\leq \frac{1}{2} \mbox{ for } i>\ell\}$ and $g$ is the Euclidean metric. The first inequality is due to mean value inequality, and the second one is Jensen inequality.  Hence
\begin{align*}
|s(z)|_{h,\omega_P}&=|s(z)|_{h(\bm{a},-N),\omega_P}\cdot (-\prod_{i=1}^{\ell}\log |z_i|^2)^{\frac{N}{2}}\cdot (\prod_{i=1}^{\ell}|z_i|^{-a_i})  \\
&\leq e^{\frac{C_1}{2}} \lVert s\rVert_{h(\bm{a},N),\omega_P} \cdot (-\prod_{i=1}^{\ell}\log |z_i|^2)^{\frac{N}{2}}\cdot (\prod_{i=1}^{\ell}|z_i|^{-a_i})\\
&\leqslant C_\delta \lVert s\rVert_{h(\bm{a},N),\omega_P}\cdot (\prod_{i=1}^{\ell}|z_i|^{-a_i-\delta})
\end{align*}
for any $\delta>0$ and some positive constant $C_\delta$ depending on $\delta$.
\end{proof}
 \subsection{A fine resolution for Dolbeault complex of Higgs bundles}\label{sec:fine}
 Let $(E,\theta,h)$ be a tame harmonic bundle on $X-D$, where $(X,\omega)$ is a compact K\"ahler manifold and $D=\sum_{i=1}^\ell D_i$ is a simple normal crossing divisor on $X$. 
  
  Let $L$ be a line bundle on $X$ equipped with a smooth Hermitian metric $h_L$
 so that $\sqrt{-1}R(h_L)\geq 0$ and has at least $n-k$ positive eigenvalues. Such a metrized line bundle $(L,h_L)$ is indeed called \emph{k-positive} in \cite{SS85}.  Let {$B$} be a nef line bundle on $X$.   Let $\sigma_i$ be the section $H^0(X,\sO_X(D_i))$ defining $D_i$, and we fix some smooth Hermitian metric $h_i$ for the line bundle $\sO_X(D_i)$ so that $|\sigma_i|_{h_i}(z)<1$ for any $z\in X$. Write $\sigma_D:=\prod_{i=1}^{\ell} \sigma_i\in H^0(X,\sO_X(D))$ and $h_D:=\prod_{i=1}^{\ell} h_i$ the smooth metric for $\sO_X(D)$.   Pick a   positive constant $N$ greater than   $N_0$, where $N_0$ is the constant in \Cref{lem:estimate} so that \eqref{eq:estimate 2}   and \Cref{prop:C0} hold  for $(E, \theta, h)$.

Given a smooth metric {$h_B$} on {$B$}, note that for $\bm{a}=(a_1,\ldots,a_\ell)\in \mathbb{R}^\ell$ and $\sL:=L\otimes {B}|_{X^*}$ equiped with the metric
 \begin{align}\label{eq:metric}
h_\sL(\bm{a}):=h_L{h_B}\prod_{i=1}^{\ell}|\sigma_i|^{2a_i}_{h_i}\cdot (-\prod_{i=1}^{\ell}\log  |\sigma_i|_{h_i}^2)^{N},
\end{align} 

 its curvature 
 \begin{align}\label{eq:curvature}
 \sn	R(h_\sL(\bm{a})) &=\sn R(h_L)+ \sn R({h_{B}})+\sum_{i=1}^{\ell} 2\sqrt{-1}a_i  R(h_i)\\\nonumber
 &+\sn N\sum_{i=1}^{\ell}\frac{\d \log \lvert \sigma_i\rvert^2_{h_i}\wedge \db \log \lvert \sigma_i\rvert^2_{h_i}}{(\log \lvert \sigma_i\rvert^2_{h_i})^2} -N\sum_{i=1}^{\ell}\frac{\sn R(h_i)}{(\log \lvert \sigma_i\rvert^2_{h_i})^2} 
 \end{align}
  Here $R(h_i)$ is the curvature of $\big(\sO_X(D_i),h_i\big)$.
 
Let $0\leq \gamma_1(x)\leq \cdots\leq \gamma_n(x)$ be eigenvalues of $\sn R(h_L)$ with respect to $\omega$. Set $$\ep_0:=\inf_{X}\gamma_{k+1}(x)$$ which is strictly positive by our assumption on $\sn R(h_L)$.

 \begin{lem}\label{lem:choice}
  We can rescale $h_i$ by timing a positive small constant, take proper metric {$h_{B}$} for {$B$} and pick   $\bm{a}\in \mathbb{R}^\ell_{>0}$   small enough so that
 \begin{enumerate}[leftmargin=0.7cm]
 	\item \label{choice 3} One has \begin{align}\label{eq:nef}
 	\sn R(h_\sL(\bm{a}))\geq \sn R(h_L)-\ep_1\omega\geq -\ep_1\omega.
 	\end{align}
 	for   $\ep_1=\frac{\ep_0}{100n^2}$.
 \item  \label{choice 4}
 The  metric \begin{align}\label{eq:poi}
 \omega_{\bm{a},N}: =   \ep_2\omega+\sn R(h_\sL(\bm{a}))  
 \end{align}\label{eq:new-kahler}
is a   K\"ahler metric  when restricted on $X^*=X-D$ for   $\ep_2=\frac{\ep_0}{10n}$.   
\item \label{choice 5} $\diae$=${}_{\bm{a}}E$.
\end{enumerate} 
\end{lem} 
\begin{proof}
	 Let us explain how to achieve \eqref{choice 3}. The possible negative contribution for $\sn R(h_\sL(\bm{a}))$ can only come  from  $$\sn R({h_{B}})+\sum_{i=1}^{\ell} 2\sqrt{-1}a_i  R(h_i)  -N\sum_{i=1}^{\ell}\frac{\sn R(h_i)}{(\log \lvert \sigma_i\rvert^2_{h_i})^2}.$$ As {$B$} is nef, one can take {$h_{B}$} so that $\sn R({h_{B}})\geq -\frac{1}{2}\ep_1\omega$.  As $N$ is fixed, we can replace $h_i$ by $c\cdot h_i$ for $0<c\ll1$ and let $a_i$'s small enough,  so that $\sum_{i=1}^{\ell} 2\sqrt{-1}a_i  R(h_i)  -N\sum_{i=1}^{\ell}\frac{\sn R(h_i)}{(\log \lvert \sigma_i\rvert^2_{h_i})^2}\geq -\frac{1}{2}\ep_1\omega$.  By \cref{thm:SM} and \cref{dfn:parab-higgs}, one has
	 $\diae$=${}_{\bm{a}}E$ if $\bm{a}$ is chosen small enough.  \eqref{choice 4} follows from \eqref{choice 3} directly.
\end{proof}

We know that $\omega_{\bm{a},N}$ is a  \emph{complete} K\"ahler metric. Indeed, write $h_i\stackrel{\text{loc}}{=}e^{-\varphi_i}$ in terms of the trivialization   $D_i\cap U=(z_i=0)$ of any admissible coordinate $(U;z_1,\ldots,z_n)$, one has
\begin{align*} 
 \omega_{\bm{a},N}  
 &=\big(\ep_2\omega+\sum_{i=1}^{\ell} 2\sqrt{-1}a_i R(h_i)+ \sn R(h_{M})\big)\\
 &+N\sum_{i=1}^{\ell}\frac{1}{(\log |z|_i^2+\varphi_i)^2}(\frac{ d z_i}{z_i}+\partial \varphi_i)\wedge (\frac{d\bar{z}_i}{\bar{z}_i}+\db \varphi_i)\\
 &-N\sum_{i=1}^{\ell}\frac{\hess \varphi_i}{\log |z|_i^2+\varphi_i}
 \end{align*}

  From this local expression one can also see that $\omega_{\bm{a},N}\sim \omega_P$ on any $U^*(r)$ for $0<r<1$. We also can show the following
\begin{lem} \label{lem:perturbe}
  For the smooth metric $h_{\bm{a},N}:=h\cdot   \prod_{i=1}^{\ell}|\sigma_i|^{2a_i}_{h_i}\cdot (-\prod_{i=1}^{\ell}\log  |\sigma_i|_{h_i}^2)^{N}$  of $E$, it  is mutually bounded with $h(\bm{a},N)$ defined in \cref{sec:dolbealt} on any $U^*(r)$ for $0<r<1$.    
\end{lem}
Let us prove that such construction satisfies the positivity condition in \cref{cor:AN}. 


\begin{prop}\label{prop:positive}
	With the above notations, for any $p+q>n+k$, one has
\begin{align}\label{eq:critical2}
\langle [\sn R(h_\sL(\bm{a})), \Lambda_{\omega_{\bm{a},N}}]f, f\rangle_{\omega_{\bm{a},N}}\geq \frac{\ep }{2}\lvert f\rvert_{\omega_{\bm{a},N}}^2
\end{align}  
for any $f\in \Lambda^{p,q}T_{X^*,x}^*$ and any $x\in X^*$.
\end{prop}
\begin{proof}
For any point $x\in X^*$, one can choose local coordinate $(z_1,\ldots,z_n)$ around $x$  so that  $\omega=\sn\sum_{i=1}^{n}  dz_i\wedge d\bar{z}_i$ and $\sn R(h_\sL(\bm{a}))=\sn\sum_{i=1}^{n}\tilde{\gamma}_i dz_i\wedge d\bar{z}_i$ at $x$,  where $ \tilde{\gamma}_1\leq \cdots\leq \tilde{\gamma}_n$ are eigenvalues of $\sn R(h_\sL(\bm{a}))$ with respect to $\omega$.  By \eqref{eq:nef} one has $\tilde{\gamma}_i\geq \gamma_i-\ep_1$.   Let $\lmd_1\leq \cdots\leq \lmd_n$ be eigenvalues of $\sn R(h_\sL(\bm{a}))$ with respect to $\omega_{\bm{a},N}$. Then  $\lmd_i=\frac{\tilde{\gamma}_i}{\ep_2+\tilde{\gamma}_i}$ by \cref{lem:choice} (\ref{eq:new-kahler}), and thus  at each point $x\in X^*$, one has
 \begin{itemize}
 	\item $  -\frac{\ep_1}{\ep_2-\ep_1}\leq \lmd_i\leq 1$ for $i=1,\ldots,n$.
 	\item $\lmd_i\geq 1-\frac{\ep_2}{\ep_0-\ep_1}$ for for $i=k+1,\ldots,n$.
 \end{itemize} 
We can assume that $p\geq q$. Then
\begin{align*}
\langle [\sn R(h_\sL(\bm{a})), \Lambda_{\omega_{\bm{a},N}}]f, f\rangle_{\omega_{\bm{a},N}}&\geq 
(\sum_{i=1}^{p}\lmd_i+\sum_{j=1}^{q}\lmd_j-\lmd_1-\cdots-\lambda_n)\lvert f\rvert_{\omega_{\bm{a},N}}^2
\\
&\geq \big((p-k)(1-\frac{\ep_2}{\ep_0-\ep_1})- \frac{k\ep_1}{\ep_2-\ep_1}-(n-q)\big)\lvert f\rvert_{\omega_{\bm{a},N}}^2 \\
&\geq \big(1-n(\frac{\ep_2}{\ep_0-\ep_1}+\frac{\ep_1}{\ep_2-\ep_1})\big)\lvert f\rvert_{\omega_{\bm{a},N}}^2
  \geq \frac{1}{2}\lvert f\rvert_{\omega_{\bm{a},N}}^2.
\end{align*} 
\end{proof}
\begin{rem}
Let us mention that 	\Cref{lem:choice} and \Cref{prop:positive} are indeed inspired by the proof of  \emph{Girbau vanishing theorem} in \cite[Chapter \rom{7}, Theorem 4.2]{Dembook} and its logarithmic generalization in \cite[Theorem 4.1]{HLWY16}.
\end{rem}

We equip  ${E}$ with the metric $h_{\bm{a},N}$ and $X^*$ with the complete K\"ahler metric $\omega_{\bm{a},N}$ having the same growth as  $\omega_P$ near $D$. 
Let $\kL_{(2)}^m(E)_{h_{\bm{a},N},\omega_{\bm{a},N}}$ be the sheaf on $X$ (rather than on $X^*$!)  of germs of locally $L_2$, $E$-valued $m$-form $u$, for which  $D''(u)$  exists weakly as locally $L^2$-form. Namely, for any open set $U\subset X$, we define
\begin{align}\label{eq:locally integrable}
\kL_{(2)}^m(E)(U):=\{u\in L^m_{(2)}(U-D,E) \mid D'' u\in L_{(2)}^{m+1}(U-D,E)  \}.
\end{align}
Here we write $\kL_{(2)}^m(E)$ instead of $\kL_{(2)}^m(E)_{h_{\bm{a},N},\omega_{\bm{a},N}}$ for short.

We  also define $\kL_{(2)}^{p,q} (E)$ to be be the sheaf on $X$   of germs of locally $L_2$, $E$-valued $(p,q)$-form, for which $\db_E(u)$ exists weakly as locally $L^2$-form. Namely, for any open set $U\subset X$, one has
\begin{align}\label{eq:locally integrable2}
\kL_{(2)}^{p,q}(E)(U):=\{u\in L^{p,q}_{(2)}(U-D,E)\mid \db_Eu\in L^{p,q+1}_{(2)}(U-D,E)\} \end{align}
Note that for  
any admissible coordinate $(U;z_1,\ldots,z_n)$, as  $\omega_{\bm{a},N}\sim \omega_P$ and $ h_{\bm{a},N}\sim h(\bm{a},N)$ on any $U^*(r)$ for $0<r<1$, we have that $L^m_{(2)}(U^*(r),E)_{h(\bm{a},N),\omega_P}$ (in \cref{eq:locally integrable}) and $L^{p,q}_{(2)}(U^*(r),E)_{h(\bm{a},N),\omega_P}$ (in \cref{eq:locally integrable2}) are the same as the ones in \cref{sec:H-est-for-Higgs}.

The following lemma is a consequence of \cref{thm:moc}.
\begin{lem}\label{lem:crucial}
 Let $(E,\theta,h)$ be a tame harmonic bundle over $X-D$.  Suppose $\theta$ has nilpotent residues on $D$. We have that
 $$\kL_{(2)}^m(E)=\oplus_{p+q=m} \kL_{(2)}^{p,q}(E)$$
 and $$\theta(\kL^{p,q}_{(2)}(E))\subset \kL^{p+1,q}_{(2)}(E)$$  
\end{lem}
\begin{proof}
 Since $\theta$ is one-form with value in $\End(E)$, its norm remains unchanged if we replace the metric $h$ by   $h(\bm{a},N):=h\cdot e^{-\chi(\bm{a},N)}$.  One thus has
	$$
	|\theta|_{h(\bm{a},N),\omega_P}=	|\theta|_{h,\omega_P}\leq C
	$$
	for some  $C>0$, where the last inequality follows from  \cref{thm:moc}. (Let us stress here that this is the only place where we use the condition that $\theta$ has nilpotent residues on $D$.) Hence $\theta$ is a bounded linear operator between Hilbert spaces
	$$
	L^{p,q}_{(2)}(U-D,E)\to L^{p+1,q}_{(2)}(U-D,E).
	$$ 
	The  theorem follows from that $D''=\db_E+\theta$ and $\db_E\theta=0$.
\end{proof}
 

\begin{prop}\label{thm:inclusion}
Let $(E,\theta,h)$ be a tame harmonic bundle over $X-D$.  For $x\in D$ and any admissible coordinate $(U;z_1,\ldots,z_n)$ centered at $x$, one has 
\begin{align}\label{eq:same}
	\Gamma(U^*(r),\Omega_{U^*(r)}^m\otimes E|_{U^*(r)})\cap \kL_{(2)}^{m,0}(E)(U(r))=\big( \Omega^m_{X}(\log D)\otimes \diae\big) (U(r))
\end{align}
	if $0<r\ll 1$. In particular, 
	\begin{align}\label{eq:same2}
		\Omega^m_{X}(\log D)\otimes \diae\subset \kL_{(2)}^{m,0}(E). 
	\end{align}
\end{prop} 
\begin{proof}

Assume that $D\cap U=(z_1\ldots z_\ell=0)$. Write $w_i=\log z_i$ for $i=1,\ldots,\ell$ and $w_j=z_j$ for $j=\ell+1,n$. For the   basis $dw_I$ of $ \Omega^{m}_{X}(\log D)$, on  $U^*(r)$ with $0<r<1$, one has
 $$
 |dw_I |_{\omega_P}\leq C_1 \prod_{i=1}^{\ell}(-\log |z_i|^2),
 $$
 for some constant $C_1$.
 
 Firstly, we proof ``$\supseteq$'' of \cref{eq:same}. Pick any section $s\in\big( \Omega^m_{X}(\log D)\otimes \diae\big) (U(r))$
 One can write $$s=\sum_{I}dw_I \otimes e_I$$ with $e_I\in\diae(U(r))$. Then 
 $$|e_I|_h\leq C_2\prod_{i=1}^{\ell}|z_i|^{-\ep}$$ for any $\ep>0$ by the definition of $\diae$.  Therefore, one has 
$$
|dw_I\otimes e_I|_{h(\bm{a},N),\omega_P}\leq  |dw_I\otimes e_I|_{h,\omega_P}\cdot e^{-\chi(\bm{a},N)}=O(\prod_{i=1}^{\ell}(|z_i|^{a_i-\ep})).
$$ for any $I$, $\ep>0$. This proves that
 $$
\int_{U^*(r)}| dw_I\otimes e_\alpha|^2_{h(\bm{a},N),\omega_P}\omega_P^n=O(1),
 $$
 and thus
 $$
\Gamma(U^*(r),\Omega_{U^*(r)}^m\otimes E|_{U^*(r)})\cap \kL_{(2)}^{m,0}(E)(U(r))\supseteq \big( \Omega^m_{X}(\log D)\otimes \diae\big) (U(r)).
 $$

Now we prove ``$\subseteq$'' of  \cref{eq:same}.  
 For any  section $s\in \Gamma(U^*(r),\Omega_{U^*(r)}^m\otimes E|_{U^*(r)})$,  we write $$s=\sum_{I}dw_I \otimes e_I$$ with $e_I\in E(U^*(r))$.  If $s\in \kL_{(2)}^{m,0}(E)(U(r))$, it follows from \Cref{prop:C0} that
 $$
 	|s|_{h,\omega_P}(z)\leq C   (\prod_{i=1}^{\ell}|z_i|^{-a_i-\delta})
 	$$
 for any $\delta>0$. Hence
 $$
C   (\prod_{i=1}^{\ell}|z_i|^{-a_i-\delta})\geq  |s|_{h,\omega_P}=\sum_{I}{|dw_I|_{\omega_P}}|e_I|_{h}\geq \sum_{I} |e_I|_{h}
 $$
 for any $\delta>0$ and $0<r\ll 1$. Therefore, one has
 $$
 e_I\in  {}_{\bm{a}}E(U(r)). 
 $$
 Since ${}_{\bm{a}}$ is chosen properly so that ${}_{\bm{a}}E=\diae$, one conclude that $$s\in\big( \Omega^m_{X}(\log D)\otimes \diae\big) (U(r)).$$
 This proves that	$$\Gamma(U^*(r),\Omega_{U^*(r)}^m\otimes E|_{U^*(r)})\cap \kL_{(2)}^{m,0}(E)(U(r))\subseteq \big( \Omega^m_{X}(\log D)\otimes \diae\big) (U(r)).$$ 
  \eqref{eq:same} follows. \eqref{eq:same2} is a consequence of   \eqref{eq:same}.
   
\end{proof} 

 Notice that in \Cref{thm:inclusion}, one does not need to assume that $\theta$ has nilpotent residues on $D$, which is essentially required in \cref{lem:crucial}. For the remaining of \cref{sec:fine}, we put this nilpotency assumption. Recall that one has $D''^2=0$. Let $(\kL^\bullet_{(2)}(E),D'')$  be a complex of fine sheaves over $X$ defined by
\begin{align} 
\kL^0_{(2)}(E)\xrightarrow{D''} \kL^1_{(2)}(E)\xrightarrow{D''}   \cdots \xrightarrow{D''}  \kL^m_{(2)}(E).
\end{align} 
By \eqref{eq:same2} and \cref{lem:crucial}, there is a natural inclusion
\begin{equation}\label{dia:quasi}
\begin{tikzcd}[row sep=1.35em, column sep=0.7em]
\diae \arrow[r,"\theta"]\arrow[d]&\diae\otimes \Omega^1_X(\log D)  \arrow[r,"\theta"]\arrow[d]&\cdots  \arrow[r,"\theta"]\arrow[d]&\diae\otimes \Omega^n_X(\log D)\arrow[d]& &\\
 \kL^0_{(2)}(E)\arrow[r,"D''"]& \kL^1_{(2)}(E)\arrow[r,"D''"]&\cdots \arrow[r,"D''"]& \kL^n_{(2)}(E)\arrow[r,"D''"]&\cdots \arrow[r,"D''"]   & \kL^{2n}_{(2)}(E)
\end{tikzcd} 
\end{equation}
and we are going to show that this morphism between two complexes is a quasi-isomorphism.

We now recall a celebrated theorem (in a weaker form)  by Demailly \cite[Th\'eor\`eme 4.1]{Dem82}, which enables us to solve the $\db$-equation on weakly pseudo-convex K\"ahler manifold (might not be complete). When the metric is complete, it is due to Andreotti-Vesentini \cite{AV65}.  
\begin{thm}[Demailly]\label{thm:Dem}
	Let $(X,\omega)$ be a K\"ahler manifold ($\omega$ might not be complete), where $X$ possesses a complete K\"ahler metric (\emph{e.g.} $X$ is weakly pseudo-convex). Let $E$ be a vector bundle on $X$ equipped with a smooth hermitian metric $h$ so that 
	$$
	\sqrt{-1}R(E,h)\geq_{\nak} \varepsilon \omega\otimes \vvmathbb{1}_E,
	$$ 
	where $\varepsilon>0$ is a positive constant. Assume that $g\in L^{n,q}_{(2)}(X,E)$ so that $\db g=0$. Then
	there exists $f\in L^{n,q-1}_{(2)}(X,E)$ so that $\db f=g$ and $$\lVert f\rVert_{h,\omega}^2\leqslant \varepsilon^{-1}\lVert g\rVert_{h,\omega}^2.$$
\end{thm}
This theorem by Demailly is used to solve the $\db$-equation locally. We first recall the notation used in the following proposition and theorem. Let $(X, \omega)$ be a compact K\"ahler manifold and $D=\sum_{i=1}^\ell D_i$ be a simple normal crossing divisor on $X$. Let $(E,\theta,h)$ be a tame harmonic bundle on $X-D$. With the modified Hermitian metric $h_{\bm{a}, N}$ for $E$ and the complete K\"ahler metric $\omega_{\bm{a}, N}$ defined in \cref{lem:choice} and \cref{lem:perturbe}, we have the sheaves of $L^2$ $E$-valued forms $\kL_{(2)}^{p,q}(E)_{h_{\bm{a},N},\omega_{\bm{a},N}}$ defined in \cref{eq:locally integrable2}.  We write $\kL_{(2)}^{p,q}(E)$ instead of $\kL_{(2)}^{p,q}(E)_{h_{\bm{a},N},\omega_{\bm{a},N}}$ for short.
\begin{prop}\label{prop:delbar}
	For any $x\in X$, there is an open set $U\subset X$ (can be made arbitrary small) containing $x$ so that for any  $g\in \kL_{(2)}^{p,q}(E)(U)$ with $q\geq 1$ and $\db_E(g)=0$, there exists a section $f\in \kL_{(2)}^{p,q-1}(E)(U)$ so that $\db_E f=g$.
\end{prop}
\begin{proof} 
	If $x\notin D$, then we can take an open set $U\subset X-D$ containing $x$ which is biholomorphic to a polydisk, and the theorem follows from the usual $L^2$-Dolbeault lemma. Assume $x\in D$.
	Let $(\tilde{U};z_1,\ldots,z_n)$ be an admissible coordinate around $x$.  By \Cref{lem:estimate}, $\sE_p:=  T_{\tilde{U}^*}^p\otimes E$   equipped with the   $\sC^\infty$-metric $ {h}_{\sE_p}=h_p^{-1}h(\bm{a},N)$  induced by  $h(\bm{a},N)$ and $\omega_P$, satisfying 
	$$
	\sqrt{-1}R(h_{\sE_p})\geq_{\nak} \omega_P\otimes \vvmathbb{1}_{\sE_p}
	$$
	for any $p=0,\ldots,n$.
	Note that $\omega_P|_{\tilde{U}^*(\frac{1}{2})}\sim \omega_{\bm{a},N}|_{\tilde{U}^*(\frac{1}{2})}$ and $ h(\bm{a},N)|_{\tilde{U}^*(\frac{1}{2})}\sim h_{\bm{a},N}|_{\tilde{U}^*(\frac{1}{2})}$. Hence  one has 
	\begin{align}\label{eq:correspondence} L_{(2)}^{n,q}(\tilde{U}^*(\frac{1}{2}),\sE_{n-p})_{h_{\sE_{n-p}},\omega_P}=L_{(2)}^{p,q}(\tilde{U}^*(\frac{1}{2}),E)_{h_{\bm{a},N},\omega_{\bm{a},N}} 
	\end{align}
	for any $p=0,\ldots,n$.
	For any $g\in L_{(2)}^{n,q}(\tilde{U}^*(\frac{1}{2}),\sE_{n-p})_{h_{\sE_{n-p}},\omega_P}$ with $\db (g)=0$, if $q\geq 1$, by \cref{thm:Dem}, there is $f\in L_{(2)}^{n,q-1}(\tilde{U}^*(\frac{1}{2}),\sE_{n-p})_{h_{\sE_{n-p}},\omega_P}$ so that $\db f=g$. The proposition then follows from \eqref{eq:correspondence},  and $\tilde{U}^*(\frac{1}{2})$ is the desired open set $U$ in the proposition.
\end{proof}
Now we are ready to prove that the $L^2$-complex is the desired fine resolution for our tame harmonic bundle.
\begin{thm}\label{thm:quasi}
	The morphism between two complexes in \eqref{dia:quasi} is a  quasi-isomorphism. 
\end{thm}
\begin{proof}
Pick any $m\in \{0,\ldots,n\}$. We are going to show that $\iota:\ker \theta/\Ima \theta\to \ker D''/\Ima D''$ at $\diae\otimes \Omega^m_X(\log D)$  is an isomorphism. For any $x\in D$, we pick an open set $U\ni x$ as in \Cref{prop:delbar} and set $U^*=U-D$. Indeed, $U^*=\tilde{U}^*(\frac{1}{2})$ where   $(\tilde{U};z_1,\ldots,z_n)$ is an admissible coordinate   around $x$ and thus $h_{\bm{a},N}\sim h(\bm{a},N)$ and $\omega_{\bm{a},N}\sim \omega_P$ on $U^*$.   Pick any $g\in \kL_{(2)}^{m}(E)(U)$  so that $D''g=0$. By \cref{lem:crucial}, we can write $g=\sum_{p+q=m}g_{p,q}$ where $g_{p,q}\in \kL_{(2)}^{p,q}(E)(U)$, and let $q_0$ be the largest integer for $q$ so that $g_{p,q}\neq 0$. By \cref{lem:crucial}, we can decompose $D''g$ into bidegrees, so that
$$
\begin{cases}
	\db_Eg_{m-q_0,q_0}=0\\
	\theta g_{m-q_0,q_0}+\db_E  g_{m-q_0+1,q_0-1}=0\\
	\vdots\\
	\theta g_{p_0-1,m-p_0+1}+\db_E  g_{p_0,m-p_0}=0\\
	\theta g_{p_0,m-p_0}=0
\end{cases}
$$
for which, the operators act in the sense of distribution. Hence $g_{m-q_0,q_0}\in \kL_{(2)}^{m-q_0,q_0}(E)(U)$ with $\db_Eg_{m-q_0,q_0}=0$. Applying \Cref{prop:delbar}, there is a section $f_{m-q_0,q_0-1}\in \kL_{(2)}^{m-q_0,q_0-1}(E)(U)$ so that $\db_E f_{m-q_0,q_0-1}=-g_{m-q_0,q_0}$. By \Cref{lem:crucial}, $D''f_{m-q_0,q_0-1}\in \kL_{(2)}^{m}(E)(U)$, and we define $g':=D''f_{m-q_0,q_0-1}+g\in \kL_{(2)}^{m}(E)(U)$. One thus has $D''g'=0$. Write $g'=\sum_{p+q=m}g'_{p,q}$ where $g'_{p,q}\in \kL_{(2)}^{p,q}(E)(U)$. Note that
$$
\begin{cases}
	g'_{m-q_0,q_0}=\db_E f_{m-q_0,q_0-1}+g_{m-q_0,q_0}=0\\
	g'_{m-q_0+1,q_0-1}=\theta f_{m-q_0,q_0-1} +g_{m-q_0+1,q_0-1} \\
	g'_{m-q_0+2,q_0-2}= g_{m-q_0+2,q_0-2}\\
	\vdots\\
	g'_{p_0,m-p_0}= g_{p_0,m-p_0}
\end{cases}
$$
One can perform the same manner inductively to find $f\in \kL_{(2)}^{m-1}(E)(U)$ so that $g_0=g+D''f\in \kL_{(2)}^{m,0}(E)(U)$ so that $D''g_0=0$. Decomposing $D''g_0$ into bidegrees we get
$$
\db(g_0)=0, \quad \theta(g_0)=0.
$$
By the elliptic regularity of $\db$ one concludes that 
$$
g_0\in \Gamma(U^*, \Omega_{U^*}^m\otimes E|_{U^*}). 
$$
By \eqref{eq:same}, $g_0\in \Gamma(U,\Omega^m_{X}(\log D)\otimes\diae|_U) $, which shows the surjectivity of $\iota$. 
 
Now we prove the injectivity of $\iota$. Let $g \in \Gamma(U,\Omega^m_{X}(\log D)\otimes\diae|_U)\subset \kL_{(2)}^m(E)(U) $ so that $g=D''f$. Write $f=\sum_{p+q=m{-1}}f_{p,q}$ where $f_{p,q}\in \kL_{(2)}^{p,q}(E)(U)$. Then $g=D''(f_{m-1,0})=\theta(f_{m-1,0})$ thanks to the bidegree condition.  Hence
$$
f_{m-1,0} \in \Gamma(U^*,\Omega_{U^*}^{m-1}\otimes E|_{U^*})\cap \kL_{(2)}^{m-1,0}(E)(U).
$$
By \eqref{eq:same} again, $f_{m-1,0}\in \Gamma(U,\Omega^{m-1}_{X}(\log D)\otimes\diae|_U) $. The injectivity of $\iota$ follows. 

When $m>n$, the exactness of $D''$ can be proven in the same way. Let $g\in \kL_{(2)}^{m}(E)(U)$ so that $D''g=0$. Applying \Cref{prop:delbar} once again as the case of $m\leq n$, we can find $f\in \kL_{(2)}^{m-1}(E)(U)$ so that $D''f+g\in \kL_{(2)}^{n,m-n}(E)(U)$. As $\theta (D''f+g)=0$, this implies that $\db_E(D''f+g)=0$, and by \Cref{prop:delbar}  again one can find $h\in \kL_{(2)}^{n,m-n-1}(E)(U)$ so that $D''h=\db_Eh=D''f+g$. This shows the exactness of $D''$ when $m>n$. We complete the proof of the theorem.
\end{proof}

\begin{rem}
To summarize, let us explain    our choice of the perturbation of the metric $h$ by $h_{\bm{a},N}:=h\cdot   \prod_{i=1}^{\ell}|\sigma_i|^{2a_i}_{h_i}\cdot (-\prod_{i=1}^{\ell}\log  |\sigma_i|_{h_i}^2)^{N}$. 
	
	The input of  the factor $\prod_{i=1}^{\ell}|\sigma_i|^{2a_i}_{h_i}$ is to assure that  the  sections of $\diae$ is  $L^2$-integrable, which seems not true for the harmonic metric $h$. However,   $a_i$'s have to be small enough since  holomorphic sections of $E$ which are also $L^2$-integrable with respect to $h_{\bm{a},N}$ only lies on ${}_{\bm{a}}E$.  Due to the semicontinuity of the parabolic structures  by Mochizuki, ${}_{\bm{a}}E=\diae$ only if $a_i$'s are small enough. This is the main context of \Cref{thm:inclusion}. 
	
	The input of $(-\prod_{i=1}^{\ell}\log  |\sigma_i|_{h_i}^2)^{N}$ is to add enough local positivity near $D$ so that one can apply the H\"ormander-Demailly $L^2$-estimate to obtain the $L^2$-Dolbeault lemma locally around $D$. This is \Cref{prop:delbar}. Let us stress here that the fact that $(E,h)$ is \emph{acceptable} due to Mochizuki is essential to perform such modification of metrics. 
\end{rem}
\subsection{Proof of the main theorem}
In this subsection, we will prove the following vanishing theorem for tame harmonic bundle.
\begin{thm}\label{thm:main}
Let $(X,\omega)$ be a compact K\"ahler manifold of dimension $n$ and let $D$ be a simple normal crossing divisor on $X$. Let $(E,{}_{ \bm{a}}E,\theta)$ be the parabolic Higgs bundle on $X$ induced by a  tame harmonic bundle  $(E,\theta,h)$   on $X^*=X-D$ whose Higgs field has nilpotent residues on $D$.   Let $L$ be a line bundle on $X$ equipped with a smooth Hermitian metric $h_L$ 
so that its curvature $\sqrt{-1}R(h_L)\geq 0$ and has at least $n-k$ positive eigenvalues at every point on $X$ as a real (1,1)-form. Let {$B$} be a nef line bundle on $X$.  Then
$$
\mathbb{H}^m\big(X, (\diae\otimes \Omega^\bullet_X(\log D),\theta)\otimes L\otimes {B} \big)=0
$$
for any $m>n+k$.
\end{thm}
\begin{proof}
	We will use the notations in  \cref{sec:fine}. 
Recall that $(X^*,\omega_{\bm{a},N})$ is a complete K\"ahler manifold. Write $\sL:=L\otimes {B}|_{X^*}$ and we equip it with the metric $g=h_L{h_{{B}}}$ where {$h_{{B}}$} is properly chosen as \cref{lem:choice}.  Then $g$ is the restriction to $X^*$ of a smooth metric on  $X$. We introduce a new Higgs bundle $(\tilde{E},\tilde{\theta},\tilde{h}):=(E\otimes \sL, \theta\otimes   \vvmathbb{1}_{\sL},h_{\bm{a},N}\cdot g)$.  We still use the notation $D'':=\db_{\tilde{E}}+\tilde{\theta}$ abusively, and $D''^*$ denotes  its adjoint with respect to $\tilde{h}$. We will apply \cref{cor:AN} to solve $D''$-equation for this new Higgs bundle.

Note that $h_{\bm{a},N}g=hh_\sL{(\bm{a})}$ by \eqref{eq:metric} and \cref{lem:perturbe}.   By \Cref{prop:positive}, the metrized line bundle $(\sL,h_\sL{(\bm{a})})$ satisfies the condition in \cref{cor:AN} when $m>n+k$. Hence   by \cref{cor:AN} for any section $g{\in} L^m_{(2)}(X^*,\tilde{E})_{\tilde{h},\omega_{\bm{a},N}}$, if $D''g=0$ and $m>n+k$, there exists $f\in L^{m-1}_{(2)}(X^*,\tilde{E})_{\tilde{h},\omega_{\bm{a},N}}$ so that 
$$
D''f=g.
$$ 
 Let $\kL_{(2)}^m(\tilde{E})_{\tilde{h},\omega_{\bm{a},N}}$ be the sheaf on $X$ (rather than on $X^*$!)  of germs of locally $L_2$, $\tilde{E}$-valued $m$-forms, for which both $D''(u)$ (as a distribution) exist weakly as locally $L^2$-forms. Namely, for any open set $U\subset X$, one has
 \begin{align}\label{eq:locally integrable3}
 \kL_{(2)}^m(\tilde{E})(U):=\{u\in L^m_{(2)}(U-D,\tilde{E})_{\tilde{h},\omega_{\bm{a},N}} \mid D'' u\in L_{(2)}^{m+1}(U-D,E)_{\tilde{h},\omega_{\bm{a},N}}  \}
 \end{align}
 Then the above argument proves that the cohomology $H^i$ of the complex of global sections of the sheaves
 $(\kL_{(2)}^\bullet(\tilde{E})_{\tilde{h},\omega_{\bm{a},N}},D'')$ vanishes for $m>n+k$. 
 
 As $g$ is smooth over the whole $X$, the metric $\tilde{h}\sim  h(\bm{a},N)$ near $D$ (fix any trivialization of $L\otimes {B}$).  Hence the  natural inclusion
 \begin{equation}\label{dia:quasi3}
 \begin{tikzcd} [row sep=1.35em, column sep=0.7em]
 \diae\otimes L\otimes  {B} \arrow[r,"\tilde{\theta}"]\arrow[d]&\diae\otimes L\otimes {B}\otimes \Omega_X(\log D)  \arrow[r,"\tilde{\theta}"]\arrow[d]&\cdots  \arrow[r,"\tilde{\theta}"]\arrow[d]&\diae\otimes L\otimes {B}\otimes \Omega^n_X(\log D)\arrow[d]& &\\
 \kL^0_{(2)}(\tilde{E})_{\tilde{h},\omega_{\bm{a},N}}\arrow[r,"D''"]& \kL^1_{(2)}(\tilde{E})_{\tilde{h},\omega_{\bm{a},N}}\arrow[r,"D''"]&\cdots \arrow[r,"D''"]& \kL^n_{(2)}(\tilde{E})_{\tilde{h},\omega_{\bm{a},N}}\arrow[r,"D''"]&\cdots \arrow[r,"D''"]   & \kL^{2n}_{(2)}(\tilde{E})_{\tilde{h},\omega_{\bm{a},N}}
 \end{tikzcd} 
 \end{equation}
 is thus also a quasi-isomorphism by \cref{thm:quasi}.  

As the complex $(\kL_{(2)}^\bullet(\tilde{E})_{\tilde{h},\omega_{\bm{a},N}},D'')$ is a fine sheaf,   its cohomology computes the hypercomology of the complex $(\diae\otimes L\otimes {B}\otimes \Omega^\bullet_X(\log D),\tilde{\theta})$. We thus conclude that $\bH^m(X,(\diae\otimes L\otimes {B}\otimes\Omega^\bullet_X(\log D),\tilde{\theta}))=0$ for $m>n+k$. The theorem is proved. 
\end{proof} 
\begin{rem}\label{rem:Yang}
	Let us show how to derive the log Girbau  vanishing theorem in \cite[Corollary 1.2]{HLWY16} from \cref{main}. {In this remark we use the same notation as that in \cite[Corollary 1.2]{HLWY16}.} With the same setting as \cref{main}, let $(E,\theta,h):=(\sO_{X-D},0,h)$ where $h$ is the canonical metric on the trivial line bundle $\sO_{X-D}$. According to the prolongation of $(E,\theta,h)$ defined in \cref{sec:prolongation}, one has $(\diae,\theta)=(\sO_X,0)$. Hence the Dolbeault complex  in \eqref{eq:dol}
	$$
	{\rm Dol}(\diae,\theta)=\sO_X\xrightarrow{0} \Omega^1_X(\log D) \xrightarrow{0}\cdots\xrightarrow{0}  \Omega^n_X(\log D) 
	$$
	which is a   direct sum of sheaves of  logarithmic $p$-forms shifting $p$ places to the right:
	$$
		{\rm Dol}(\diae,\theta)=\oplus_{p=0}^{n}\Omega^p_X(\log D)[p],
	$$
	where $\Omega^p_X(\log D)[p]$ {is obtained} by shifting the \emph{single degree complex} $\Omega^p_X(\log D)$ in degree $p$.
	Hence if $m>n+k$,	by \cref{thm:main} one has
	\begin{align*}
0=\mathbb{H}^m\big(X, {\rm Dol}(\diae,\theta)\otimes {N\otimes L} \big) 
&=\oplus_{p=0}^{n}H^m(X,\Omega^p_X(\log D)\otimes {N\otimes L}[p])\\
&=\oplus_{p=0}^{n}H^{m-p}(X,\Omega^p_X(\log D)\otimes {N\otimes L}). 
\end{align*}
We thus  conclude that 
	$$
	H^{q}(X,\Omega^p_X(\log D)\otimes {N\otimes L})
	$$
	if $p+q>n+k$. This is  the log Girbau  vanishing theorem  by Huang-Liu-Wan-Yang.
\end{rem}


\subsection{Vanishing theorem for parabolic Higgs bundles}\label{sec:parabolic}
Let $X$ be a complex projective manifold  and let $D$ be simple normal crossing divisor on $X$. 
For a parabolic Higgs bundle $(E,{}_{\bm{a}}E,\theta)$ on    $(X,D)$, its parabolic Chern classes, denoted by $\textnormal{para}$-$c_i(E)$, is the usual Chern class of $\diae$ with a modification along the boundary divisor $D$ (see, e.g., \cite[\S 3]{AHL19} for more details). With a polarization, i.e., an ample line bundle $H$ on $X$, the parabolic degree $\textnormal{para}$-$\textnormal{deg}(E)$ of $(E,{}_{\bm{a}}E,\theta)$ is defined to be $\textnormal{para-}c_1(E)\cdot H^{n-1}$. We say $(E,{}_{\bm{a}}E,\theta)$ \emph{slope stable}  if  for any coherent
torsion free subsheaf $V$ of $\diae$, with $0<\textnormal{rank}
V<\textnormal{rank} \diae=\textnormal{rank} E$ and $\theta(V)\subseteq V\otimes
\Omega^1_X(\log D)$, the condition 
$$   \frac{\textnormal{para-deg}(V)}{\rank(V)}< \frac{\textnormal{para-deg}(E)}{\rank(E)} $$ 
is satisfied, where $V$ carries the induced the parabolic structure from $(E,{}_{\bm{a}}E,\theta)$, i.e. ${}_{\bm{a}}V:=V\cap {}_{\bm{a}}E.$ A parabolic Higgs bundle $(E,{}_{\bm{a}}E,\theta)$ is \emph{poly-stable} if it is a direct sum of slope stable parabolic Higgs bundles. By \cite{IS07},  $(E,{}_{\bm{a}}E,\theta)$ is called \emph{locally abelian} if in a Zariski neighborhood of any point $x\in X$ there is an isomorphism between the underlying parabolic vector bundle  $(E,{}_{\bm{a}}E)$
and a direct sum of parabolic line bundles. 

By the celebrated   Simpson-Mochizuki correspondence  \cite[Theorem 9.4]{Moc06},  a parabolic Higgs bundle $(E,{}_{\bm{a}}E,\theta)$ on $(X,D)$ is poly-stable with trivial parabolic {Chern classes} and    locally abelian if and only if it is   induced by a tame harmonic bundle over $X-D$ defined in \cref{sec:prolongation}. Based on this deep theorem, our theorem can thus be restated as follows. 
\begin{cor} \label{cor:main}
 Let $(E, {}_{\bm{a}}E,\theta)$ be a locally abelian poly-stable parabolic Higgs bundle on  a projective log pair $(X,D)$ with trivial parabolic {Chern classes} so that  the Higgs field $\theta$ has nilpotent residues on $D$.  Let $L$ be a line bundle on $X$ equipped with a smooth metric $h_L$ 
so that its curvature $\sqrt{-1}R(h_L)\geq 0$ and has at least $n-k$ positive eigenvalues. Let {$B$} be a nef line bundle on $X$.   Then  for the weight 0 filtration $\diae$  of $(E,{}_{\bm{a}}E,\theta)$, one has
$$
\mathbb{H}^m\big(X, (\diae\otimes \Omega^\bullet_X(\log D),\theta)\otimes L\otimes {B} \big)=0
$$
for any $m>\dim X+k$.  
\end{cor} 
\begin{rem}
The above corollary essentially generalizes the main theorem \cite[Theorem 1]{Ara19} in which he assumed that $\theta$   is  nilpotent (see \cref{rmk:nilpotency-compare}) and that $L$ is ample. 
\end{rem}

\end{document}